\documentclass{article}

\usepackage[numbers]{natbib}
     \usepackage[preprint]{neurips_2022}

\usepackage[utf8]{inputenc} 
\usepackage[T1]{fontenc}    
\usepackage{hyperref}       
\usepackage{url}            
\usepackage{booktabs}       
\usepackage{amsfonts}       
\usepackage{nicefrac}       
\usepackage{microtype}      
\usepackage{xcolor}         
\usepackage{todonotes}

\usepackage{makecell}
\usepackage{caption}
\usepackage{multirow}

\newcommand{\cbraces}[1]{\left( #1 \right)}
\newcommand{\braces}[1]{\left\{ #1 \right\}}
\usepackage{colortbl}
\definecolor{bgcolor}{rgb}{0.8,1,1}
\definecolor{bgcolor2}{rgb}{0.8,1,0.8}
\usepackage{threeparttable}
\newcommand{\myred}[1]{{\color{red}#1}}
\newcommand{\myblue}[1]{{\color{blue}#1}}

\usepackage{amsmath,amssymb,amsthm}
\usepackage{mathtools}
\usepackage{algorithm}
\usepackage{algpseudocode}
\usepackage{cleveref}
\crefname{assumption}{assumption}{assumptions}
\allowdisplaybreaks[3]

\newtheorem{theorem}{Theorem}
\newtheorem{lemma}{Lemma}

\newtheorem{remark}{Remark}
\newtheorem{assumption}{Assumption}

\newcommand{\R}{\mathbb{R}}

\def\<#1,#2>{\langle #1,#2\rangle}

\newcommand{\norm}[1]{\|#1\|}
\newcommand{\sqn}[1]{\norm{#1}^2}

\newcommand{\cO}{\mathcal{O}}

\newcommand{\eqdef}{\coloneqq}
\DeclareMathOperator*{\argmin}{arg\,min}

\title{Optimal Gradient Sliding and its Application to Distributed Optimization Under Similarity}

\author{%
  Dmitry Kovalev\\
  KAUST\\
  \texttt{dakovalev1@gmail.com}
  \And
  Aleksandr Beznosikov\\
  MIPT\\
  \texttt{anbeznosikov@gmail.com}
  \And
  Ekaterina Borodich\\
  MIPT\\
  \texttt{borodich.ed@phystech.edu}
  \And
  Alexander Gasnikov\\
  MIPT\\
  \texttt{gasnikov@yandex.ru}
  \And
  Gesualdo Scutari\\
  Purdue University \\
  \texttt{gscutari@purdue.edu}
}

\begin{document}

\maketitle

\begin{abstract}
We study structured convex  optimization problems, with additive objective   $r:=p + q$, where $r$ is ($\mu$-strongly) convex, $q$ is $L_q$-smooth and convex, and $p$ is $L_p$-smooth,   possibly nonconvex. For such a class of problems, we proposed   an inexact accelerated gradient sliding   method  that 
can skip the gradient computation for one of these   components while still achieving optimal   complexity of gradient calls of $p$ and $q$, that is, 
 $\mathcal{O}(\sqrt{L_p/\mu})$   and $\mathcal{O}(\sqrt{L_q/\mu})$, respectively.   This result is much sharper than the classic black-box  complexity $\mathcal{O}(\sqrt{(L_p+L_q)/\mu})$,   especially when  the difference between $L_q$ and $L_q$ is large. We then apply the proposed method to solve distributed optimization problems over master-worker architectures, under agents' function similarity, due to statistical data similarity or
otherwise. The distributed algorithm achieves for the first time lower complexity bounds on {\it both} communication and local  gradient calls, with the former having being a long-standing open problem. Finally the method is extended to distributed saddle-problems (under function similarity)   by means of solving a class of variational inequalities, achieving lower communication and computation   complexity bounds.
\end{abstract}

\section{Introduction}
We consider structured convex programming in the form   \cite{beck2009fast,duchi2010composite,nesterov2013gradient}:
\begin{equation}
\label{eq:main}
	\min_{x \in \R^d}  r(x) \eqdef q(x) + p(x),
\end{equation}
where  $r$ is assumed to be  convex and decomposed as the sum of a smooth, possibly nonconvex  function $p$ and a smooth convex function $q$. First order information of $p$ and $q$ is accessible separately. We are interested in scenarios where the cost of evaluating   the gradient of the two functions is not even, but computing, say  $\nabla p$, is much more resource demanding than $\nabla q$. The motivating application for this scenario is   distributed optimization  over master-worker systems, as discussed next.

Consider the following distributed optimization problem over a network of $m$ agents:
\begin{equation}\label{eq:main2}
	\min_{x\in\R^{d}} r(x) = \frac{1}{n}\sum_{i=1}^n f_i(x),
\end{equation}
where $f_i$ is the loss function of agent $i$, assumed to be convex, which is not known to other agents. Agents are embedded in a star-topology, with agent $1$ being the master node, without loss of generality--this is the typical federated learning setup \cite{kairouz2021advances}.  An instance of (\ref{eq:main2}) of particular interest is the empirical risk minimization (ERM) whereby the goal is to minimize the average loss over some dataset, distributed across the nodes of the network, with $f_i$ being the empirical risk of agent $i$, i.e.,   $f_i(x)=\frac{1}{m}\sum_{j=1}^m \ell\big(x;z_i^{j}\big)$, 
  where   $z_i^{(1)},\ldots, z_i^{(m)}$ is the set of $m$ samples owned by agent $i$, and   $\ell(x;z_i^{j})$  measures the mismatch between the   parameter $x$ and the sample $z_i^{j}$.
  
Several solutions methods have been proposed to solve (\ref{eq:main2}); the prototype approach consists in interleaving local computations at the workers sides (nodes $i=1,\ldots n$) with communications to/from the master node ($i=1$), 
which maintains and updates the authoritative copy of the optimization variables, producing eventually the final solution estimate.  
Since the cost of communications is often  the bottleneck in distributed computing  (e.g., \cite{Bekkerman_book11,Lian17}), a lot of research has been devoted to designing distributed algorithms that are \emph{communication efficient}. Acceleration (in the sense of Nesterov) has been extensively investigated as a procedure to reduce the communication burden.  For $L$-smooth and $\mu$-strongly convex functions $r$ in (\ref{eq:main2}), linear convergence    is certified by employing first-order methods, with  computation (gradient evaluations) and communication complexities  proportional to  $\sqrt{\kappa}$   ($\kappa\triangleq L/\mu$ is the condition number of $r$). For ill-conditioned functions ($\kappa$ very large), the polynomial dependence on $\kappa$  may be   unsatisfactory.  This is, e.g., the typical setting of many ERM problems wherein the optimal regularization parameter   for test predictive performance is very small.   
   
   Further improvements on the communication complexity can be obtained exploiting the extra structure typical in ERM problems, also known as   {\it function similarity} (see, e.g.,  \cite{arjevani2015communication, pmlr-v37-zhangb15, stich2018local, woodworth2020local}): 
  $\|\nabla^2 f_i(x) - \nabla^2 f_j\|\leq  {\delta}$,
 for all $x$ in a proper domain of interest and all   $i\neq j=1,\ldots, n$, where $\delta>0$ measures the degree of similarity between the  Hessian matrices of  the  local  losses. When data are i.i.d. among agents,    $f_i$'s   reflect  statistical similarities in local data, resulting in  $\delta=\tilde{O}(1/\sqrt{m})$ with high-probability   ($\tilde{O}$ hides log-factors and dependence on  $d$). In this scenario, in general,  $1+\delta/\mu\ll\kappa$ \cite{arjevani2015communication}. This motivated a surge of studies aiming at exploiting function similarity coupled with acceleration to boost communication efficiency  (see Sec.~\ref{sec:related_works} for an overview of relevant works):   linear convergence is certified with a  number of communication steps (for nonquadratic losses) scaling   with $\widetilde{\mathcal{O}}( \sqrt{\delta/\mu})$, where $\widetilde{\mathcal{O}}$ hides log-factors. This matches   lower (communication)  complexity bounds \citep{arjevani2015communication} {\it only up to log-factors}. Furthermore, these methods   are not computationally optimal, yielding  local gradients calls larger than lower complexity bounds $\mathcal{O}(\sqrt{\kappa})$. In fact,   Table~\ref{tab:comparison0} shows that, to the date, there exists no distributed algorithm   achieving the best of the two worlds, that is, optimal (lower bound)  communication complexity {\it and}   local gradient (oracle) complexity.  
 
 This paper fills this gap. 
 Our starting point is the reformulation of (\ref{eq:main2}) in the  equivalent form    \vspace{-0.1cm} \begin{equation}\label{eq:pq}
	\min_{x\in\R^{d}}  r(x) = \underbrace{f_1(x)}_{:=q(x)} +  \underbrace{\frac{1}{n}\sum_{i=1}^n [f_i(x) - f_1(x)]}_{:=p(x)},
\end{equation}
 which  exploits function similarity at the agents' side via preconditioning.  Problem (\ref{eq:pq}) is an instance of (\ref{eq:main}):  all $f_i$ (thus $q$) are convex but  $p$ is   nonconvex. Also, evaluating $\nabla q$ and $\nabla p$ has  different costs; the former involves only local computations at the master node while the latter requires communications from/to master and workers nodes. At high-level the idea is then clear:   one would like to design a distributed algorithm for (\ref{eq:pq}) [or more generally for \eqref{eq:main}] that skips      gradient computations of $\nabla p$ (saving thus communications) without slowing down the overall optimal rate of convergence. 
 
 This naturally suggests the use of {\it gradient-sliding} techniques 
  \cite{gorbunov2019optimal, sadiev2021decentralized, borodich2021decentralized}, yielding algorithms that 
   skip from time to time   computation of the gradient of one function in the summand objective. However, existing gradient-sliding algorithms are not applicable to   \eqref{eq:main} [and thus (\ref{eq:pq})]  because  they all require $p$ and $q$ to be {\it convex}.  This calls for new designs, accounting for the nonconvexity of $p$. \vspace{-0.2cm}

\subsection{Main contributions}\vspace{-0.1cm}
Our contribution is threefold: 

$\bullet$ \textbf{A new gradient-sliding algorithm for \eqref{eq:main}:}  We propose     a new Accelerated ExtraGradient sliding method that skips the computation of  $\nabla p$   from time to time. 
The method builds on an inexact  acceleration of a  proximal envelop (outer-loop) coupled with a suitable termination criterion and inner-loop algorithm to approximately solve the proximal subproblem. When applied to  \eqref{eq:main}, with   $r$ being  ($\mu$-strongly) convex,  $q$ being $L_q$-smooth and convex, and  $p$ being  $L_p$-smooth (possibly nonconvex), the proposed  algorithm achieves   optimal complexity of  gradient calls of $q$ and $p$, that is, 
\renewcommand{\arraystretch}{2}
\begin{table*}[!h]
    \centering    
    \label{tab:1}   
    \scriptsize
    \begin{tabular}{|c|c|c|}
    \hline
    $r$& $\nabla q$ & $\nabla p$\\\hline 
    strongly convex & $\mathcal{O}\left( \sqrt{\frac{L_q}{\mu}} \log \frac{1}{\varepsilon} \right)$ & $\mathcal{O}\left( \sqrt{\frac{L_p}{\mu}} \log \frac{1}{\varepsilon} \right)$\\
    \hline 
    convex & $\mathcal{O}\left( \sqrt{\frac{L_q}{\varepsilon}} \|x^0 - x^* \| \right)$ & $\mathcal{O}\left( \sqrt{\frac{L_p}{\varepsilon}} \|x^0 - x^* \| \right)$ 
    \\\hline 
    \end{tabular}
\vspace{-0.5cm}
\end{table*}

Notice that the above complexity bounds are sharper than the complexity bound obtained by the  Nesterov's optimal first-order method for smooth (strongly) convex optimization applied to \eqref{eq:main}. For instance for strongly convex $r$, that would yield    $\nabla q$, $\nabla p$ complexity scaling as $\mathcal{O}(\sqrt{(L_q+L_p)/\mu})$, which is less favorable than our   separate complexity bounds above. To the best of our knowledge, this is the first time that such bounds are achieved for nonconvex $p$.

$\bullet$ \textbf{Optimal complexity bounds for \eqref{eq:main2}  under function similarity: } We   customize the proposed accelerated gradient-sliding algorithm to the distributed optimization problem \eqref{eq:main2} under $\delta$-function similarity. 
As showed in Table~\ref{tab:comparison0} for strongly convex $r$  (see   Sec.~\ref{sec:similar} for the case of weakly convex $r$),  the new  distributed algorithm achieves lower complexity   bounds on {\it both} the number of communications   \cite{arjevani2015communication} and   on the number of   gradient computations (without   logarithmic factors!) \cite{nesterov2018lectures}. Achieving optimal communication complexity (for nonquadratic losses)  was  a long-standing   open problem. 

$\bullet$ \textbf{Gradient sliding for variational inequalities:}  We extend the proposed  gradient-sliding machinery  to solve distributed saddle-points under similarity by means of solving a class of strongly-monotone Variational Inequalities (VI). We improve  existing complexity bounds for such problems \cite{beznosikov2021distributed} achieving for the first time both optimal communication complexity and  gradient oracle complexity--see Table~\ref{tab:comparison0}.\vspace{-0.2cm}


\renewcommand{\arraystretch}{2}
\begin{table*}[h]
    \centering
    \small
	\caption{Existing convergence results for  distributed (saddle point)  optimization   under $\delta$-similarity.}
    \label{tab:comparison0}   
    \scriptsize 
\resizebox{\linewidth}{!}{
  \begin{threeparttable}
    \begin{tabular}{|c|c|c|c|c|c|c|}
    \cline{3-7}
    \multicolumn{2}{c|}{}
     & \textbf{Reference} & \textbf{Communication complexity} & \textbf{Local gradient complexity} & \textbf{Order}  & \textbf{Limitations} \\
    \hline
    \multirow{13}{*}{\rotatebox[origin=c]{90}{\textbf{Minimization \quad\quad}}} & \multirow{11}{*}{\rotatebox[origin=c]{90}{\textbf{Upper}\quad\quad}}
    & \texttt{DANE} \cite{pmlr-v32-shamir14} & $\mathcal{O} \left( \myred{\frac{\delta^2}{\mu^2}} \log \frac{1}{\varepsilon} \right)$ & \myred{---} \tnote{{\color{blue}(2)}} & 1st & \myred{quadratic} 
    \\ \cline{3-7}
    && \texttt{DiSCO} \cite{pmlr-v37-zhangb15} & $\mathcal{O} \left( \sqrt{\frac{\delta}{\mu}} (\log\frac{1}{\varepsilon} + \myred{C^2 \Delta F_0})\myred{\log\frac{L}{\mu}}\right)$  & $\mathcal{O} \left( \sqrt{\frac{\myblue{\delta}}{\mu}} (\log\frac{1}{\varepsilon} + \myred{C^2 \Delta F_0})\myred{\log\frac{L}{\mu}}\right)$ & \myred{2nd}   & \myred{$C$ - self-concordant} \tnote{{\color{blue}(3)}}
    \\ \cline{3-7}
    && \texttt{AIDE} \cite{reddi2016aide} & $\mathcal{O} \left( \sqrt{\frac{\delta}{\mu}} \log \frac{1}{\varepsilon} \myred{\log \frac{L}{\delta}}\right)$ & $\mathcal{O} \left( \sqrt{\frac{\delta}{\mu}} \myred{\sqrt{\frac{L}{\mu}}} \log \frac{1}{\varepsilon} \myred{\log \frac{L}{\delta}}\right)$ \tnote{{\color{blue}(4)}} & 1st  & \myred{quadratic}
    \\ \cline{3-7}
    && \texttt{DANE-LS} \cite{yuan2019convergence} & $\mathcal{O} \left( \myred{\frac{\delta}{\mu}} \log \frac{1}{\varepsilon} \right)$ & $\mathcal{O} \left( \sqrt{\frac{L}{\mu}} \myred{\frac{\delta^{3/2}}{\mu^{3/2}}} \log \frac{1}{\varepsilon} \right)$ \tnote{{\color{blue}(5)}} & 1st/\myred{2nd}  & \myred{quadratic} \tnote{{\color{blue}(6)}}
    \\ \cline{3-7}
    && \texttt{DANE-HB} \cite{yuan2019convergence} & $\mathcal{O} \left( \sqrt{\frac{\delta}{\mu}} \log \frac{1}{\varepsilon} \right)$ & $\mathcal{O} \left( \sqrt{\frac{L}{\mu}} \myred{\frac{\delta}{\mu}} \log \frac{1}{\varepsilon} \right)$ \tnote{{\color{blue}(5)}} & 1st/\myred{2nd}  & \myred{quadratic} \tnote{{\color{blue}(6)}}
    \\ \cline{3-7}
    && \texttt{SONATA} \cite{sun2019distributed} & $\mathcal{O} \left( \myred{\frac{\delta}{\mu}} \log \frac{1}{\varepsilon} \right)$ & \myred{---} \tnote{{\color{blue}(2)}} & 1st & decentralized
    \\ \cline{3-7}
    && \texttt{SPAG} \cite{hendrikx2020statistically} & $\mathcal{O} \left( \sqrt{\myred{\frac{L}{\mu}}}\log \frac{1}{\varepsilon} \right)$ \tnote{{\color{blue}(1)}} &  \myred{---} \tnote{{\color{blue}(2)}} & 1st  & \myred{$M$ - Lipshitz hessian}
    \\ \cline{3-7}
    && \texttt{DiRegINA} \cite{daneshmand2021newton} & $\mathcal{O} \left( \myred{\frac{\delta}{\mu}} \log\frac{1}{\varepsilon} + \myred{\sqrt{\frac{M \delta R_0}{\mu}}}\right)$  & \myred{---} \tnote{{\color{blue}(2)}} & \myred{2nd} & \myred{$M$ -Lipshitz hessian}
    \\ \cline{3-7}
    && \texttt{ACN} \cite{agafonov2021accelerated} & $\mathcal{O} \left( \sqrt{\frac{\delta}{\mu}} \log\frac{1}{\varepsilon} + \myred{\sqrt[3]{\frac{M \delta R_0}{\mu}}}\right)$ & \myred{---} \tnote{{\color{blue}(2)}} &\myred{2nd} & \myred{$M$ -Lipshitz hessian}
    \\ \cline{3-7}
    && \texttt{Acc SONATA} \cite{tian2021acceleration} & $\mathcal{O} \left( \sqrt{\frac{\delta}{\mu}} \log \frac{1}{\varepsilon} \myred{\log \frac{\delta}{\mu}}\right)$ & \myred{---} \tnote{{\color{blue}(2)}} & 1st & decentralized
    \\ \cline{3-7}
    && \cellcolor{bgcolor2}{This paper}  & \cellcolor{bgcolor2}{$\mathcal{O} \left( \sqrt{\frac{\delta}{\mu}} \log \frac{1}{\varepsilon} \right)$} & \cellcolor{bgcolor2}{$\mathcal{O} \left( \sqrt{\frac{L}{\mu}} \log \frac{1}{\varepsilon} \right)$} & \cellcolor{bgcolor2}{1st} & \cellcolor{bgcolor2}{}
    \\ \cline{2-7}
    & \multirow{2}{*}{\rotatebox[origin=c]{90}{\textbf{Lower}}} 
    & \cite{arjevani2015communication} & $\mathcal{O} \left( \sqrt{\frac{\delta}{\mu}} \log \frac{1}{\varepsilon} \right)$& --- &&
    \\ \cline{3-7} 
    && \cite{nesterov2018lectures} & --- & $\mathcal{O} \left( \sqrt{\frac{L}{\mu}} \log \frac{1}{\varepsilon} \right)$ && non-distributed
    \\\hline\hline 
    \multirow{4}{*}{\rotatebox[origin=c]{90}{\textbf{Saddles}\quad}} & \multirow{2}{*}{\rotatebox[origin=c]{90}{\textbf{Upper}}} 
    & \texttt{SMMDSA} \cite{beznosikov2021distributed} & $\mathcal{O} \left( \frac{\delta}{\mu} \log \frac{1}{\varepsilon} \right)$ & $\mathcal{O} \left( \frac{L}{\mu} \log \frac{1}{\varepsilon} \myred{\log \frac{L}{\mu}} \right)$ & 1st &
    \\ \cline{3-7}
    && \cellcolor{bgcolor2}{ This paper}  & \cellcolor{bgcolor2}{$\mathcal{O} \left( \frac{\delta}{\mu} \log \frac{1}{\varepsilon} \right)$} & \cellcolor{bgcolor2}{$\mathcal{O} \left( \frac{L}{\mu} \log \frac{1}{\varepsilon} \right)$} & \cellcolor{bgcolor2}{1st} & \cellcolor{bgcolor2}{}
    \\ \cline{2-7}
    & \multirow{2}{*}{\rotatebox[origin=c]{90}{\textbf{Lower}}} 
    & \cite{beznosikov2021distributed} & $\mathcal{O} \left( \frac{\delta}{\mu} \log \frac{1}{\varepsilon} \right)$& --- &&
    \\ \cline{3-7}
    && \cite{ouyang2021lower} & - & $\mathcal{O} \left( \frac{L}{\mu} \log \frac{1}{\varepsilon} \right)$ && non-distributed
    \\\hline 
    \end{tabular}   
    \begin{tablenotes}
    {\small
    \item [] \tnote{{\color{blue}(1)}} This is the worst-case complexity, as   pointed out in the paper;   the convergence of the method might be better, based upon  an additional sequence $G_t$  \cite{hendrikx2020statistically};
    \tnote{{\color{blue}(2)}} proximal local computations (exact solution of local subproblems);
    \tnote{{\color{blue}(3)}} from Lipschitzness of the Hessian and strong convexity follows self-concordance;
    \tnote{{\color{blue}(4)}} gradient  complexity not provided, we derived  it using \cite{nesterov2018lectures};
    \tnote{{\color{blue}(5)}}  gradient complexity not provided, we derived  it using \cite{nesterov2020primal};
    \tnote{{\color{blue}(6)}} gradient complexity holds for nonquadratic functions;
    \item [] {\em Notation:} $\delta=$ similarity parameter, $L$=smoothness constant of $f_i$, $\mu=$ strong convexity constant of   $r$, $\varepsilon=$accuracy of the solution, $R_0 := \|x^0 - x^* \|$, $\Delta F_0 :+ r(x^0) - r(x^*)$.     
    }
\end{tablenotes}    
    \end{threeparttable}
    }\vspace{-0.4cm}
\end{table*}

\subsection{Related works}\label{sec:related_works}\vspace{-0.2cm}
\textbf{Gradient-Sliding:} Since the seminal paper \cite{lan2016gradient}, the idea of gradient-sliding    for structured convex optimization such as \eqref{eq:main} has received significant attention   as tool to skip gradient computations  of one function in the summand; examples and generalization include  first-order accelerated methods   \cite{lan2016gradient,lan2016accelerated,lan2016conditional},
zero-order (derivative-free) schemes \cite{dvinskikh2020accelerated,ivanova2020oracle,beznosikov2020derivative,stepanov2021one},   
high-order methods \cite{kamzolov2020optimal,gasnikov2021accelerated,ahookhosh2021high,grapiglia2022adaptive}, and  slidings for  saddle point problems and variational inequalities \cite{alkousa2020accelerated,lan2021mirror,tominin2021accelerated,beznosikov2021distributed}.
Albeit applicable to  more general classes of optimization problems  than \eqref{eq:main} (e.g., allowing either $p$ and $q$ to be nonsmooth), none of the existing methods provide guarantees when $p$ is nonconvex. 
The proposed algorithm fills this gap. Furthermore, it achieves optimal lower complexity bounds on the calls of $\nabla p$ and $\nabla q$. 

\textbf{Distributed optimization under function similarity:} The literature of distributed optimization is vast;   given the focus of this work,   we   comment  next solution  methods   exploiting function similarity  via proper preconditioning--Table~\ref{tab:comparison0} summarizes  complexity results of existing distributed methods solving either  minimization problems or saddle-point formulations, and is commented next.

The seminal paper \cite{arjevani2015communication} established lower communication complexity bounds for  (\ref{eq:main2}) under $\delta$-similarity: $\varepsilon$-optimality cannot be achieve in less than  $\Omega(\sqrt{{\delta}/{\mu}} \log {1}/{\varepsilon})$ communication rounds. Since then, 
  a lot of effort has been devoted to design     distributed schemes 
  aiming at   achieving   optimal communication complexity. The authors in \cite{pmlr-v32-shamir14} proposed  \texttt{DANE},    a   mirror-descent based algorithm     whereby   workers   perform a local data preconditioning via a suitably chosen Bregman divergence, and the master averages the solutions of the workers. For {\it quadratic} losses, \texttt{DANE} achieves communication complexity $\widetilde{\mathcal{O}}((\delta/\mu)^2\log1/\varepsilon)$; this was later improved to   ${\mathcal{O}}((\delta/\mu)\log1/\varepsilon)$    for nonquadratic losses   in    \cite{sun2019distributed}, where   the \texttt{SONATA} algorithm was proposed (also implementable over mesh-networks).  
 
 Improvements were achieved employing   acceleration;  efforts include:  \texttt{DiSCO} \cite{pmlr-v37-zhangb15}, an inexact
damped Newton method coupled with a preconditioned conjugate gradient (to compute the Newton direction), which   achieves  communication complexity  $\widetilde{\mathcal{O}}(\sqrt{\delta/\mu})\log1/\varepsilon)$ for self-concordant losses (see Table \ref{tab:comparison0} for the log-factors hidden in the $\widetilde{\mathcal{O}}$); \texttt{AIDE} \cite{reddi2016aide}, which uses the Catalyst framework  \cite{lin2018catalyst}, matching the rate of  \texttt{DiSCO} for quadratic losses; \texttt{DANE-HB} \cite{yuan2019convergence}, a variant of  \texttt{DANE} equipped with Heavy Ball momentum and matching for quadratic functions the communication complexity of \texttt{DiSCO} and \texttt{AIDE}; and \texttt{SPAG} \cite{hendrikx2020statistically}, a preconditioned direct accelerated method,    achieving for nonquadradic losses {\it asymptotically} the convergence rate   ${\mathcal{O}}((1-1/\sqrt{\beta/\mu})^k)$   ($k$ is the iteration index)--the worst-case rate is still ${\mathcal{O}}(\sqrt{L/\mu})\log 1/\varepsilon)$.

Finally, higher order methods employing preconditioning have been studied in   \cite{daneshmand2021newton,agafonov2021accelerated,tian2021acceleration}: \cite{daneshmand2021newton} proposed \texttt{DiRegINA},  a decentralization of  the cubic regularization of the Newton method,  where workers build Newton direction  sampling local   Hessians;  \cite{agafonov2021accelerated}  introduced \texttt{ACN}, an   inexact accelerated cubic-regularized Newton's method, with improved complexity with respect to   \cite{daneshmand2021newton}; and    \cite{tian2021acceleration} extended the Catalyst framework \cite{lin2018catalyst} to the distributed setting (including mesh networks), proposing  
\texttt{Acc SONATA}--the communication complexity of these methods is reported in Table \ref{tab:comparison0}.  

In summary, the above tour on  the relevant literature shows that none of the existing methods can match  lower communication complexity bounds  for (non quadratic) optimization problems (\ref{eq:main2}) under function similarity (all complexity bounds contain log-factors). The proposed distributed method achieves lower communication and computation complexity bounds.

 \vspace{-0.2cm}

\section{Optimal Gradient Sliding for Minimization Problems} \label{sec:slid_min}\vspace{-0.1cm}

We study the  minimization problem \eqref{eq:main},  under the following blanket assumptions. 
   
\begin{assumption}\label{ass:mu}
    $r(x)\colon \R^d \rightarrow \R$ is $\mu$-strongly convex on $\mathbb{R}^d$. 
\end{assumption}

\begin{assumption}\label{ass:L_q}
    $q(x)\colon \R^d \rightarrow \R$ is convex and $L_q$-smooth on $\mathbb{R}^d$. 
\end{assumption}

\begin{assumption}\label{ass:L_p}
    $p(x)\colon \R^d \rightarrow \R$ is $L_p$-smooth on $\R^d $.
\end{assumption}


The proposed   Accelerated ExtraGradient sliding is formally introduced in Algorithm \ref{ae:alg}. Convergence of the outer loop is established in Theorem \ref{ae:thm}, while Theorem \ref{aux:thm} determines complexity of solving the inner loop up to a suitable termination. Finally Theorem \ref{th3} provides the overall complexity merging inner and outer loop results. The proof of all the theorems   can be found in   Appendix \ref{sec:slid_min_strcon}. 

\begin{algorithm}[H]
	\caption{Accelerated Extragradient}
	\label{ae:alg}
	\begin{algorithmic}[1]
		\State {\bf Input:} $x^0=x_f^0 \in \R^d$
		\State {\bf Parameters:} $\tau \in (0,1)$, $\eta,\theta,\alpha > 0, K \in \{1,2,\ldots\}$
		\For{$k=0,1,2,\ldots, K-1$}
			\State $x_g^k = \tau x^k + (1-\tau) x_f^k$\label{ae:line:1}
			\State $x_f^{k+1} \approx \argmin_{x \in \R^d}\left[ A_\theta^k(x) \eqdef p(x_g^k) + \<\nabla p(x_g^k),x - x_g^k> + \frac{1}{2\theta}\sqn{x - x_g^k} + q(x)\right]$\label{ae:line:2}
			\State $x^{k+1} = x^k + \eta\alpha (x_f^{k+1}  - x^k)- \eta \nabla r(x_f^{k+1})$\label{ae:line:3}
		\EndFor
		\State {\bf Output:} $x^K$
	\end{algorithmic}
\end{algorithm}

\begin{theorem}\label{ae:thm}
	 Consider  \Cref{ae:alg} for Problem \ref{eq:main} under Assumption \ref{ass:mu}-\ref{ass:L_p}, with the following tuning: 
    $$ \textstyle
    \tau = \min\left\{1,\frac{\sqrt{\mu}}{2\sqrt{L_p}}\right\}, \quad \theta = \frac{1}{2L_p}, \quad  \eta = \min \left\{\frac{1}{2\mu},\frac{1}{2\sqrt{\mu L_p}}\right\}, \quad \alpha = \mu;
    $$  and let $x_f^{k+1}$   in \cref{ae:line:2} satisfy    
	\begin{equation}\label{aux:grad}
	\textstyle	\sqn{\nabla A_\theta^k(x_f^{k+1})} \leq  \frac{L_p^2}{3}\sqn{x_g^k- \argmin_{x \in \R^d} A_\theta^k(x)}.
	\end{equation}
    Then,  for any 
	\begin{equation}\label{eq:K}
	\textstyle	K\geq 2\max \left\{1, \sqrt{\frac{L_p}{\mu}}\right\}\log \frac{\sqn{x^0 - x^*}
		+\frac{2\eta}{\tau}\left[R(x^0) - R(x^*)\right]}{\varepsilon},
	\end{equation}
	we have the following estimate for the distance to the solution $x^*$:
	\begin{equation}\label{eq:accuracy}
	\textstyle	\sqn{x^K - x^*} \leq \varepsilon.
	\end{equation}
\end{theorem}

\subsection{Solving the auxiliary subproblems} \label{sec:auxprob}
	At each iteration of \Cref{ae:alg},  one needs to  solve the   subproblem:
\begin{equation}\label{aux:main}
	\textstyle \min_{x \in \R^d} A_\theta^k(x) \eqdef p(x_g^k) + \<\nabla p(x_g^k),x - x_g^k> + \frac{1}{2\theta}\sqn{x - x_g^k} + q(x).
\end{equation}

According to \Cref{ae:thm}, (\ref{aux:main}) need not be solved to arbitrary precision; inexact solutions    $x_f^{k+1}$   satisfying  condition~\eqref{aux:grad} suffice.  
Condition \eqref{aux:grad} means that gradient norm $\norm{\nabla A_\theta^k(x_f^{k+1})}$ should be sufficiently small. Notice that     $A_\theta^k(x)$ in (\ref{aux:main}) is $(2L_p + L_q)$-smooth and convex.

Problem \ref{aux:main} can be solved up to the termination \eqref{aux:grad} using any of the algorithms in  \cite{kim2021optimizing,kim2018generalizing,nesterov2020primal}. We obtain the following complexity.   
\begin{theorem}[\cite{nesterov2020primal} Remark 1]\label{aux:thm}
	There exists a certain algorithm such that, when applied  to   problem~\eqref{aux:main} with the starting point $x_g^k$, returns $x_f^{k+1}$  satisfying 
	\begin{equation}
	\textstyle	\norm{\nabla A_\theta^k(x_f^{k+1})} \leq \dfrac{D^2\cdot\max\{L_p,L_q\}\norm{x_g^k- \argmin_{x \in \R^d} A_\theta^k(x)}}{T^2},
	\end{equation}
	where $D>0$ is some universal  constant (independent of $L_p, L_q$, $T$ etc) and   $T$ is the number of calls of $\nabla q$  by the algorithm.
\end{theorem}


\subsection{Overall complexity of the optimal gradient-sliding} \label{sec:compl_min}
  \Cref{aux:thm} suggests that, to satisfy condition  \eqref{aux:grad} in \Cref{ae:thm}, it is sufficient to   choose  the number $T$ of  iterations   of the inner algorithm  as
\begin{equation}\label{eq:T}
	\textstyle T = \sqrt[4]{3}D\max\left\{1, \sqrt{\frac{L_q}{L_p}}\right\}.
\end{equation}
 
 We can now determine the overal complexity of Algorithm 1.  At each iteration of \Cref{ae:alg} we call $\nabla p$ twice (at $x^k_g$ -- line \ref{ae:line:2}
 and at $x^{k+1}_f$ -- line \ref{ae:line:3}), and $\nabla q$ is computed $T + 1$ times ($T$ times in the auxiliary problem -- line \ref{ae:line:2}, and at $x^{k+1}_{f}$ -- line \ref{ae:line:3}).
Hence, to find an $\varepsilon$-solution of  problem~\eqref{eq:main}, i.e., to find $x^K \in \R^d$ that satisfies \eqref{eq:accuracy}, \Cref{ae:alg} requires $K$   iterations as given in \eqref{eq:K}, 
\begin{equation*}
	\textstyle 2 \times K = \cO\left(\max\left\{1, \sqrt\frac{L_p}{\mu}\right\}\log\frac{1}{\epsilon}\right) ~~ \text{calls of $\nabla p(x)$,   and}
\end{equation*}
\begin{equation*}
	\textstyle (T+1) \times  K = \cO\left(\max\left\{1, \sqrt\frac{L_q}{L_p},\sqrt\frac{L_p}{\mu}, \sqrt\frac{L_q}{\mu}\right\}\log\frac{1}{\epsilon}\right) ~~\text{calls of $\nabla q(x)$.}
\end{equation*}

Putting everything together we obtain the following final convergence result.  
\begin{theorem} \label{th3}
Consider Problem~\eqref{eq:main} under   \Cref{ass:mu,ass:L_p,ass:L_q},  with $\mu \leq L_p \leq L_q$, without loss of generality. Then, to reach an $\varepsilon$-solution,  \Cref{ae:alg} requires
	\begin{equation*}
	\textstyle	\cO\left(\sqrt{\frac{L_q}{\mu}}\log\frac{1}{\epsilon}\right) ~~ \text{calls of}\,\, \nabla q(x) ~~{\text{and} }\quad 
		\cO\left(\sqrt{\frac{L_p}{\mu}}\log\frac{1}{\epsilon}\right) ~~ \text{calls of}\,\, \nabla p(x).
	\end{equation*}
\end{theorem}
This matches optimal complexity for the individual gradient calls. 

We conclude this section providing convergence of variant of the proposed algorithm suitable for    weakly convex $r$ in ~\eqref{eq:main}  (Assumption \ref{ass:mu} with $\mu = 0$). The algorithm is described in Appendix  \ref{sec:slid_min_conv}. Here we   only provide the final convergence result, the analogous  of Theorem \ref{th3}.
\begin{theorem} \label{th4}
Consider Problem~\eqref{eq:main} under Assumptions \ref{ass:mu} (with $\mu = 0$)-
\ref{ass:L_p}, with   $L_p \leq L_q$. Then, to find an $\varepsilon$-solution of  ~\eqref{eq:main} (in objective value), \Cref{ae_conv:alg} in Appendix~\ref{sec:slid_min_conv}  requires
	\begin{equation*}
	\textstyle	\cO\left(\sqrt{\frac{L_q}{\varepsilon}} \|x^0 - x^* \|\right) ~~ \text{calls of $\nabla q(x)$ ~~and}\quad 
		\cO\left(\sqrt{\frac{L_p}{\varepsilon}} \|x^0 - x^* \|\right)  ~~ \text{calls of $\nabla p(x)$}.
	\end{equation*}
	  
\end{theorem}

\section{Application to Distributed Optimization Under Similarity}\label{sec:similar}

In this section, we apply the proposed algorithm   the the distributed optimization problem \eqref{eq:main2}, under the following assumptions.

\begin{assumption}\label{ass:L}
	Each $f_i(x) \colon \R^d \rightarrow \R $ is convex and $L$-smooth.
\end{assumption}
\begin{assumption}\label{ass:mu2}
	$r(x) \colon \R^d \rightarrow \R$ is $\mu$-strongly convex.
\end{assumption}
\begin{assumption}\label{ass:delta}
	$f_1(x)\ldots, f_n(x)$ are $\delta$-related: $ \|\nabla^2 f_i(x) - \nabla^2 f(x)\| \leq  \delta$, for all $i$ and   $x \in \R^d $, and some $\delta>0$. 
  
\end{assumption}

We leverage now Algorithm 1 to solve \eqref{eq:main2}, using the  equivalent reformulation   \eqref{eq:pq}.  The algorithm applied to the distributed system can be described as follows.  The server computes $x^k_g$ and sends it to all the workers (line \ref{ae:line:1}). Workers  compute $\nabla f_i(x^k_g)$ and send it to the server. After collecting   all $\nabla f_i(x^k_g)$, the server   builds  $\nabla p(x^k_g) = \nabla f(x^k_g) - \nabla f_1 (x^k_g)$,  and then solves (inexactly) the local problem $A^k_{\theta}$ (line \ref{ae:line:2}). The inexact solution $x^{k+1}_f$ is then   broadcast    to the workers, which update their own  receives $\nabla f_i(x^{k+1}_f)$ and send back to the server, which can then evaluate $\nabla r(x^{k+1}_f)$ (line \ref{ae:line:3}).

 Using  Assumptions \ref{ass:L} and \ref{ass:mu2}, we infer   that $r$ is $\mu$-strongly convex; and  $q = f_1$ is $L_q$-smooth and convex, with $L_q = L$. It follows from   Assumption \ref{ass:delta} that  $\|\nabla^2 p \| \leq \delta$. Therefore,   $p$ has $L_p$-Lipschitz gradient, with $L_p = \delta$. This shows that we can leverage  Theorems \ref{th3} and \ref{th4} to establish convergence for   strongly convex and weakly convex $r$, as given next.    
  
\begin{theorem}
	Let \Cref{ass:L,ass:mu2,ass:delta} be satisfied with $\mu \leq \delta \leq L$.
	Then, to find $\varepsilon$-solution of the distributed optimization problem~\eqref{eq:main2} \Cref{ae:alg} requires 
	$$\textstyle \cO\left(\sqrt{\frac{\delta}{\mu}}\log\frac{1}{\epsilon} \right) ~\text{communication rounds and }~ \cO\left(\sqrt{\frac{L}{\mu}}\log\frac{1}{\epsilon} \right) ~\text{local gradient computations.}$$
\end{theorem}
\begin{theorem}
	Let Assumptions \ref{ass:L}, \ref{ass:mu2} (with $\mu = 0$), \ref{ass:delta} be satisfied and $\delta \leq L$.
	Then, to find $\varepsilon$-solution of the distributed optimization problem~\eqref{eq:main2} \Cref{ae:alg} requires 
	$$\textstyle \cO\left(\sqrt{\frac{\delta}{\varepsilon}}\|x^0 - x^* \| \right) ~\text{communication rounds and }~ \cO\left(\sqrt{\frac{L}{\varepsilon}}\|x^0 - x^* \| \right) ~\text{local computations.}$$
\end{theorem}

Such estimates are optimal from both communications \cite{arjevani2015communication} and local computations point of views \cite{nesterov2018lectures}. It is important to remark  that  Algorithm \ref{ae:alg} solves the local subproblem $A_\theta$ with some precision, while most of existing  works (see Table \ref{tab:comparison0})  assume that   local problems are solved with infinite precision (column Local gradient complexity), which is not   practical.
Note also that the subproblems in line \ref{ae:line:2} of Algorithm \ref{ae:alg} do not necessarily have to be solved by a deterministic algorithm as in Theorem \ref{aux:thm}. Stochastic methods can also be used, as long as they guarantee that condition \eqref{aux:grad} is met.\vspace{-0.2cm}

\section{Optimal Gradient Sliding for VIs}\label{sec:slid_VI}

In this section we consider the composite variational inequality   \cite{VIbook2003,Heinz} in the form:
\begin{equation} \label{eq:mainVI}
	\textstyle \text{Find } x^* \in \R^d ~:~ R(x^*)= 0 \text{ with } R(x) \eqdef Q(x) + P(x),
\end{equation}
where $Q(x), P(x) \colon \R^d \rightarrow \R^d$. Variation inequalities are a unified umbrella for a variety of problems--two examples follow.  

\textbf{Example 1 [Minimization].} Consider   problem \eqref{eq:main}, choose $Q(x) = \nabla q(x)$ and $P(x) = \nabla p(x)$. Then the solution of the variational inequality \eqref{eq:mainVI} means that we need to find the point $x^*$ where the operator $R(x^*) = \nabla r(x^*)$. For the convex function $r$, this is equivalent to finding the minimum.

\textbf{Example 2 [Saddle point problems].} Consider the convex-concave saddle point problem
\begin{align}
\label{eq:minmax}
\textstyle \min_{y \in \R^{d_y}} \min_{z \in \R^{d_z}} r(y,z):=q(y,z) + p (y, z).
\end{align}
If we take $Q(x) \eqdef Q(y,z) = [\nabla_y q(y,z), -\nabla_z q(y,z)]$ and $P(x) \eqdef P(y,z) = [\nabla_y p(y,z), -\nabla_z p(y,z)]$, then it can be proved that $x^* = (y^*, z^*)$ is a solution for \eqref{eq:mainVI} if and only if $x^* = (y^*, z^*)$ is a solution for \eqref{eq:minmax}, i.e.
$$
\textstyle r(y^*,z) \leq r(y^*,z^*) \leq r(y,z^*) \quad \text{for all } y \in \R^{d_y} \text{ and } z \in \R^{d_z}.
$$
While minimization problems are widely considered separately from variational inequalities, saddle point problems are often analysed under the VI lens.   In recent years the popularity of saddles has grown, this is due to the fact that they have both classical \cite{facchinei2007finite} and new ML \cite{goodfellow2014generative, Madry2017:adv} applications.

We study problem \eqref{eq:mainVI} under the following assumptions.  
\begin{assumption}\label{ass:muVI}
$R(x)$ is $\mu$-strongly monotone: $\langle R(x_1) - R(x_2), x_1 - x_2\rangle \geq \mu \|x_1 - x_2 \|^2$, for all $x_1, x_2 \in \R^d $.
\end{assumption}

\begin{assumption}\label{ass:L_qVI}
$Q(x)$ is monotone and $L_q$-Lipschitz: $\langle Q(x_1) - Q(x_2), x_1 - x_2\rangle \geq 0$ and $\| Q(x_1) - Q(x_2)\| \leq L_q\| x_1 - x_2\|$, for all $x_1, x_2 \in \R^d $.
\end{assumption}

\begin{assumption}\label{ass:L_pVI}
$P(x)$ is $L_p$-Lipschitz: $\| P(x_1) - P(x_2)\| \leq L_p\| x_1 - x_2\|$ for all $x_1, x_2 \in \R^d $.
\end{assumption}

For saddle point problems these assumptions are equivalent to (strong) convexity--(strong) concavity and Lipschitzness of gradients. 

\subsection{Sliding via Extragradient}

This algorithm is a non-accelerated version of Algorithm \ref{ae:alg}. A similar non-accelerated sliding is used in \cite{beznosikov2021distributed}. Our version however   has better theoretical and practical guarantees because of the effective stopping criterion \eqref{aux:gradVI}. Convergence of the outer loop is established in Theorem \ref{ae:thmVI}; Theorem \ref{aux:thmVI} establishes convergence of the inner loop up to the required termination; and finally Theorem \ref{th9} combine the two-loop complexity.  The proofs of the theorems can be found in Appendix~\ref{sec:slid_VI_strmon}. 

\begin{algorithm}[H]
	\caption{Extragradient Sliding for VIs}
	\label{ae:algVI}
	\begin{algorithmic}[1]
		\State {\bf Input:} $x^0 \in \R^d$
		\State {\bf Parameters:} $\eta,\theta,\alpha > 0, K \in \{1,2,\ldots\}$
		\For{$k=0,1,2,\ldots, K-1$}
			\State Find $u^k \approx \tilde u^k$ where $\tilde u^k$ is a solution for
			\vspace{-0.2cm}
            \begin{equation*}
            \text{Find } \tilde u^k \in \R^d ~:~ B_\theta^k(\tilde u^k)= 0 \text{ with } B_\theta^k(x) \eqdef P(x^k) + Q(x) + \frac{1}{\theta} (x - x^k)
            \end{equation*}
            \vspace{-0.5cm}
			\label{ae:line:2VI}
			\State $x^{k+1} = x^k + \eta\alpha (u^k  - x^k)- \eta R(u^k)$\label{ae:line:3VI}
		\EndFor
		\State {\bf Output:} $x^K$
	\end{algorithmic}
\end{algorithm}

\begin{theorem}\label{ae:thmVI}
Consider \Cref{ae:algVI} for Problem \eqref{eq:mainVI} under Assumptions \ref{ass:muVI}-\ref{ass:L_pVI}, with the following tuning:
$$\textstyle
    \theta = \frac{1}{2L_p}, \quad  \eta = \min \left\{\frac{1}{4\mu},\frac{1}{4 L_p}\right\}, \quad \alpha = 2\mu.
$$
 	Assume that   $u^k$ (\cref{ae:line:2VI})  satisfies 
	\begin{equation}\label{aux:gradVI}
	\textstyle	\sqn{B_\theta^k(u^k)} \leq  \frac{L_p^2}{3}\sqn{x^k- \tilde u^k}.
	\end{equation}
Then, for any 
	\begin{equation}\label{eq:KVI}
	\textstyle	K\geq 2\max \left\{1, \frac{L_p}{\mu}\right\}\log \frac{\| x^0 - x^*\|^2}{\varepsilon},
	\end{equation}
we have the following estimate for the distance to the solution $x^*$:
	\begin{equation}\label{eq:accuracyVI}
	\textstyle	\sqn{x^k - x^*} \leq \varepsilon.
	\end{equation}
\end{theorem}
The proof is given in Appendix \ref{sec:slid_VI_strmon}.

\subsection{Solving the auxiliary problem}

As for Algorithm \ref{ae:alg},   we need an auxiliary solver for the subproblem in line \ref{ae:line:2VI}, which ensures that \eqref{aux:gradVI} is met. One can observe that $B_\theta^k(x)$ is $(2L_p + L_q)$-Lipschitz and monotone. For this type of problem, the authors in \cite{yoon2021accelerated} proposes an approach that guarantees convergence on $\norm{B_\theta^k(u^k)}^2$.
\begin{theorem}[\cite{yoon2021accelerated} Corollary 2]\label{aux:thmVI}
	There exists a certain algorithm that, applied  to the subproblem~\eqref{aux:gradVI} with   starting point $x^k$, returns  $u^k$  satisfying   
	\begin{equation}
	\textstyle	\norm{B_\theta^k(u^k)}^2 \leq \frac{D^2\cdot\max\{L_p^2,L_q^2\}\norm{x^k- \tilde u^k}^2}{T^2},
	\end{equation}
	where $D>0$ is some universal numerical constant (independent of $L_p, L_q$, $T$ etc) and $T$ is the number of calls of the function $Q$.
\end{theorem}


\subsection{Complexity of the optimal gradient sliding}
    Leveraging Theorems \ref{ae:thmVI} and \ref{aux:thmVI} while following the same   reasoning as in Section \ref{sec:compl_min}, we obtain the following convergence (inner plus outer loops) for Algorithm \ref{ae:algVI}.   
\begin{theorem} \label{th9}
	Let \Cref{ass:muVI,ass:L_pVI,ass:L_qVI} be satisfied with $\mu \leq L_p \leq L_q$. Then, \Cref{ae:algVI} requires
	\begin{equation*}
	\textstyle	\cO\left({\frac{L_q}{\mu}}\log\frac{1}{\epsilon}\right) ~~ \text{calls of $Q(x)$ ~~and}\quad 
		\cO\left({\frac{L_q}{\mu}}\log\frac{1}{\epsilon}\right) ~~ \text{calls of $P(x)$}
	\end{equation*}
 to find an $\varepsilon$-solution of   problem~\eqref{eq:mainVI}.
\end{theorem}
We also consider the case of monotone VIs (Assumption \ref{ass:muVI} with $\mu = 0$). For this we modify Algorithm \ref{ae:algVI} as in   Appendix \ref{sec:slid_VI_mon} and obtain the following convergence result. 
\begin{theorem} \label{th10}
	Let Assumption \ref{ass:muVI} (with $\mu = 0$), \ref{ass:L_qVI}, \ref{ass:L_pVI} be satisfied and $L_p \leq L_q$. Then, \Cref{ae:algVIm},  described in  Appendix \ref{sec:slid_VI_mon}, requires
	\begin{equation*}
	\textstyle	\cO\left(\frac{L_q}{\varepsilon} \|x^0 - x^* \|^2\right) ~~ \text{calls of $Q(x)$ ~~and}\quad 
		\cO\left(\frac{L_p}{\varepsilon} \|x^0 - x^* \|^2\right)  ~~ \text{calls of $P(x)$}
	\end{equation*}
	 to find an $\varepsilon$-solution of the problem~\eqref{eq:mainVI}. Here $\varepsilon$-solution is measured by the value of the gap function.
\end{theorem}

\subsection{Application to distributed saddle-point problem under similarity} \label{sec:slid_VI_sim}

We apply now Algorithm \ref{ae:algVI} to solve a distributed saddle-point problem under statistical similarity, as introduced in 
\cite{beznosikov2021distributed}:
\begin{equation}\label{eq:main2spp}
	\textstyle \min_{y\in\R^{d_y}} \max_{z\in\R^{d_z}} f(y,z) \eqdef \frac{1}{n}\sum_{i=1}^n f_i(y,z).
\end{equation}
\begin{assumption}\label{ass:Lspp}
	Each $f_i(y,z) \colon \R^{d_y} \times \R^{d_z} \rightarrow \R $ is convex-concave and $L$-smooth on $\R^{d_y} \times \R^{d_z}$. 
\end{assumption}
\begin{assumption}\label{ass:mu2spp}
	$f(z,y)$ is $\mu$-strongly convex (first argument)--$\mu$-strongly concave (second argument). 
\end{assumption}
\begin{assumption}\label{ass:deltaspp}
	$f_1(z,y)\ldots, f_n(y,z)$ are $\delta$-related: for all $i$ and for all $y \in \R^{d_y} $ and $z \in \R^{d_z}$,
	\begin{equation*}
	\textstyle
	    \|\nabla^2_{yy} f_i(y,z) - \nabla^2_{yy} f(y,z)\| \leq  \delta, \|\nabla^2_{yz} f_i(y,z) - \nabla^2_{yz} f(y,z)\| \leq  \delta, \|\nabla^2_{zz} f_i(y,z) - \nabla^2_{zz} f(y,z)\| \leq  \delta.
	\end{equation*}
\end{assumption}
Casting (\ref{eq:main2spp}) into the VI formulation (\ref{eq:mainVI}), by taking  $Q(x) = Q(y,z) = [\nabla_y f_1(y,z), -\nabla_z f_1(y,z)]$ and $P(x) = P(y,z) = [\nabla_y [f-f_1](y,z), -\nabla_z [f-f_1](y,z)]$, we have that   $Q$ is monotone and $L$-Lipschitz, $P$ is $\delta$-Lipschitz, $R$ is $\mu$-strongly monotone. Therefore, we can apply Theorems \ref{th9} and \ref{th10} and obtain the following convergence results for Algorithm \ref{ae:algVI} applied to (\ref{eq:main2spp}).
\begin{theorem}
	Let \Cref{ass:Lspp,ass:mu2spp,ass:deltaspp} be satisfied with $\mu \leq \delta \leq L$.
	Then, to find $\varepsilon$-solution of the distributed saddle  problem~\eqref{eq:main2spp}, \Cref{ae:algVI} requires 
	$$ \textstyle \cO\left(\frac{\delta}{\mu}\log\frac{1}{\epsilon} \right) ~\text{communication rounds and }~ \cO\left(\frac{L}{\mu}\log\frac{1}{\epsilon} \right) ~\text{local gradient computations.}$$
\end{theorem}
\begin{theorem}
	Let Assumptions \ref{ass:Lspp}, \ref{ass:mu2spp} (with $\mu = 0$), \ref{ass:deltaspp} be satisfied and $\delta \leq L$.
	Then, to find $\varepsilon$-solution of the distributed saddle  problem~\eqref{eq:main2spp}, \Cref{ae:algVI} requires 
	$$ \textstyle \cO\left(\frac{\delta}{\varepsilon}\|x^0 - x^* \|^2 \right) ~\text{communication rounds and }~ \cO\left(\frac{L}{\varepsilon}\|x^0 - x^* \|^2 \right) ~\text{local computations.}$$
\end{theorem}

Our communication estimates are optimal, as in \cite{beznosikov2021distributed}, but our algorithm also achieve  optimal local complexity (see Table \ref{tab:comparison0}).

\section{Experiments}

\subsection{Minimization} \label{sec:exp_min}

We consider the  Ridge Regression problem
\begin{align}\label{eq:regr}
    \textstyle \min_{w} \dfrac{1}{2N} \sum\limits_{i=1}^N (w^T x_i  - y_i)^2 + \frac{\lambda}{2} \| w\|^2,
\end{align}
where $w$ is the vector of weights of the model, $\{x_i, y_i\}_{i=1}^N$ is the training dataset, and $\lambda>0$ is the regularization parameter. We consider a network with 25 workers (simulated on a single-CPU machine), and use   two types of datasets, namely: synthetic and real data. Synthetic data permit to control the similarity constant $\delta$. To do so,  we generate data on the server, say  $\{\hat x_i, \hat y_i\}_{i=1}^{n=100}$. Data   on the workers are generated by adding unbiased Gaussian noise to the server data. The lower the variance of this noise, the more similar the data, and thus the smaller  $\delta$. For simulations with real data, we considered the LIBSVM library \cite{chang2011libsvm}. The regularization parameters is set to $\lambda = 0,1$. We compare the proposed algorithm with state-of-the-art schemes, namely: DANE, DANE-HB, Accelerated gradient descent (AcGD), and SPAG. The settings of the methods are made as described in the original papers. For algorithms that assume an absolutely accurate solution of local problems (DANE, SPAG), we use AcGD with an accuracy of $10^{-12}$ as a subsolver.
Results are summarized in  Figure \ref{fig:min}--the first two figures from the left correspond to synthetic data while the other two on real data.   

\begin{figure}[h]
\begin{minipage}{0.24\textwidth}
  \centering
\includegraphics[width =  \textwidth ]{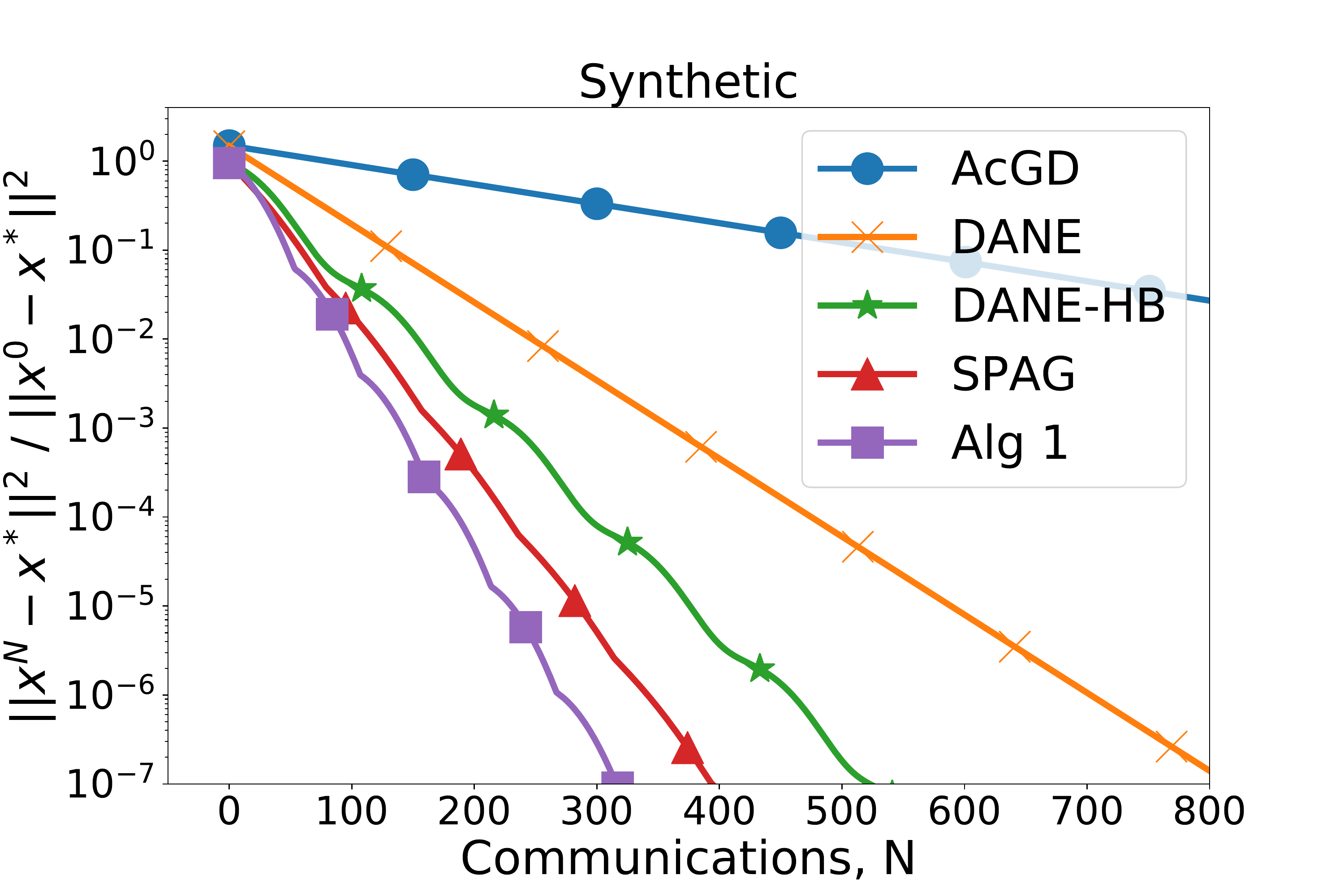}
\end{minipage}%
\begin{minipage}{0.24\textwidth}
  \centering
\includegraphics[width =  \textwidth ]{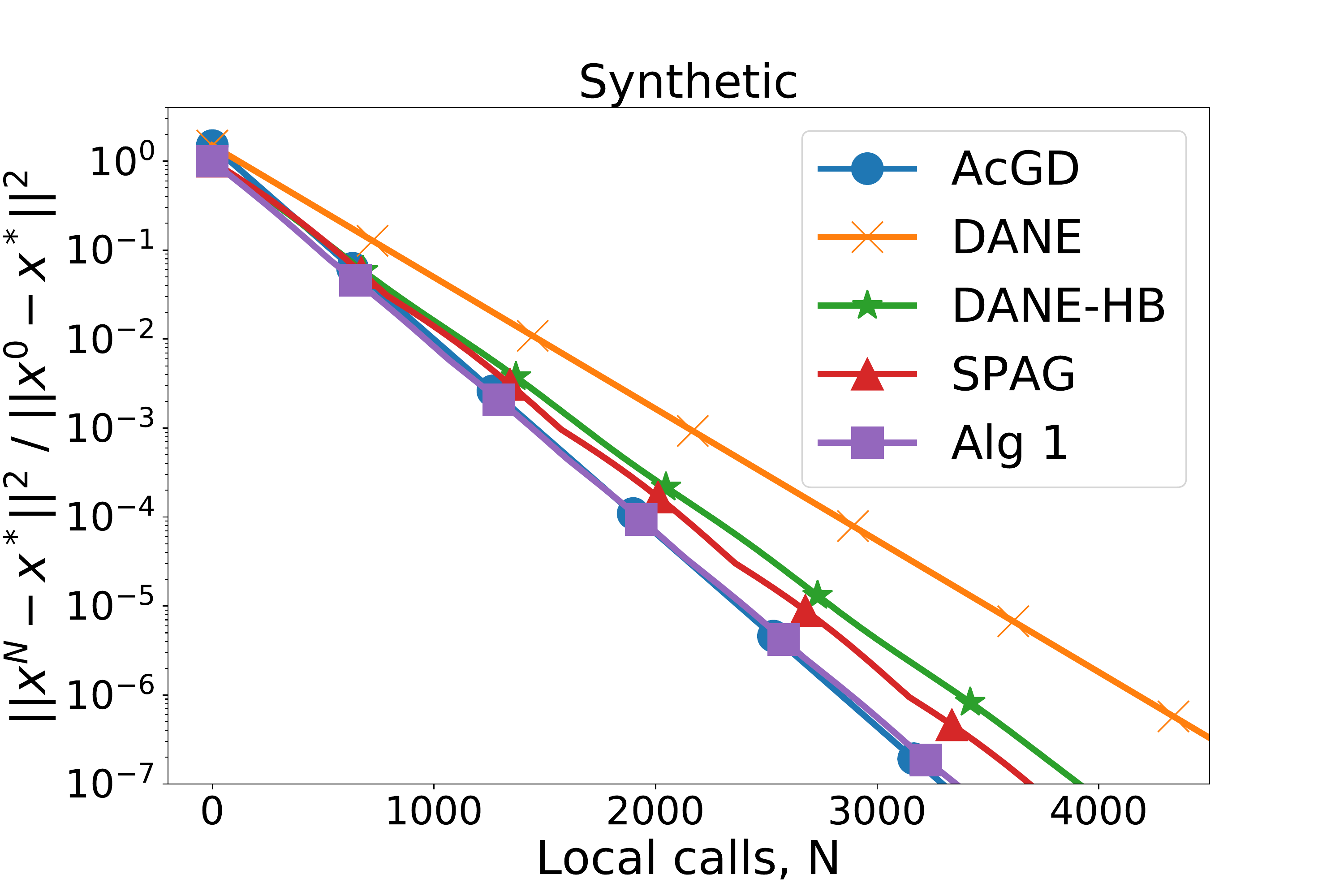}
\end{minipage}%
\begin{minipage}{0.24\textwidth}
  \centering
\includegraphics[width =  \textwidth ]{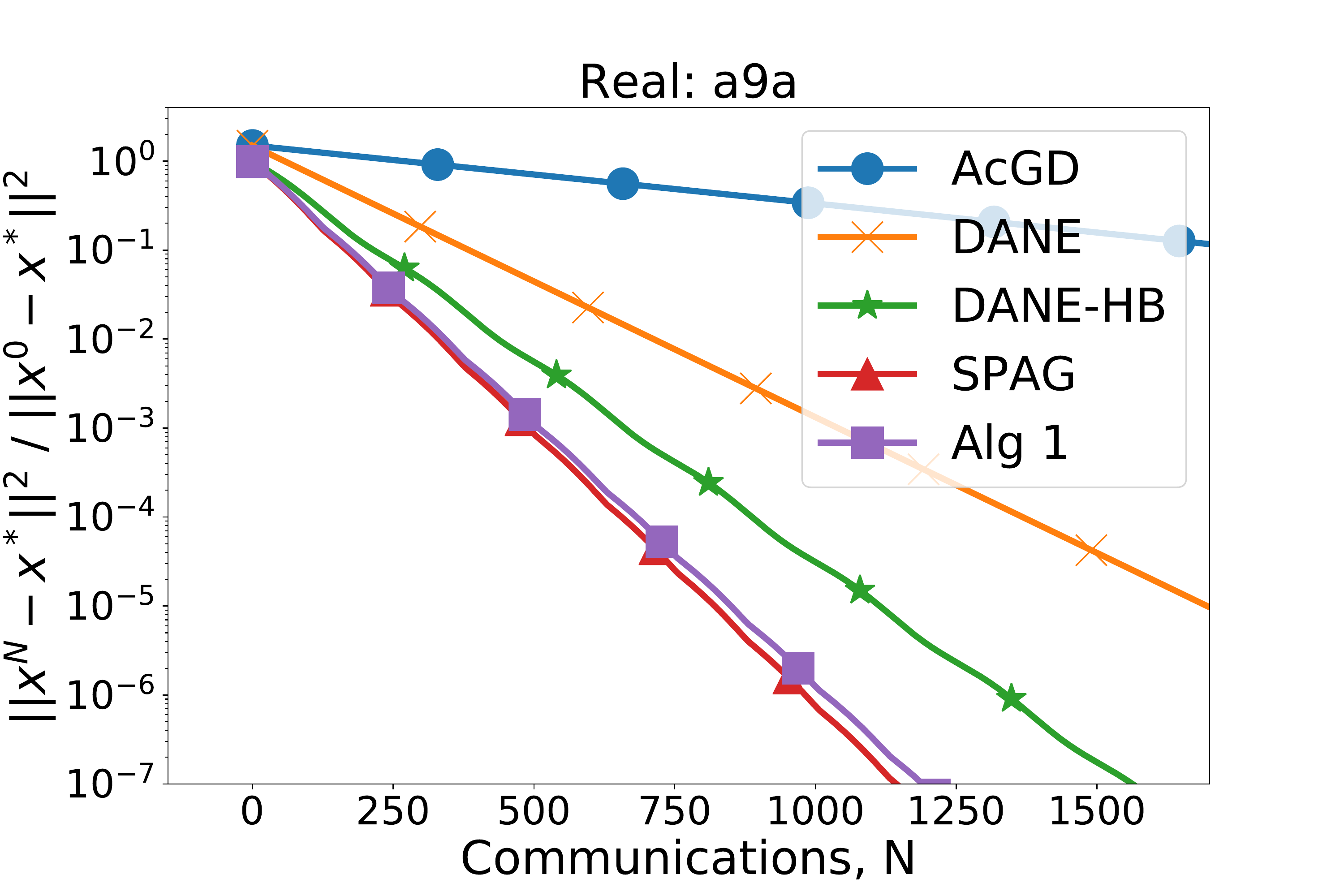}
\end{minipage}%
\begin{minipage}{0.24\textwidth}
  \centering
\includegraphics[width =  \textwidth ]{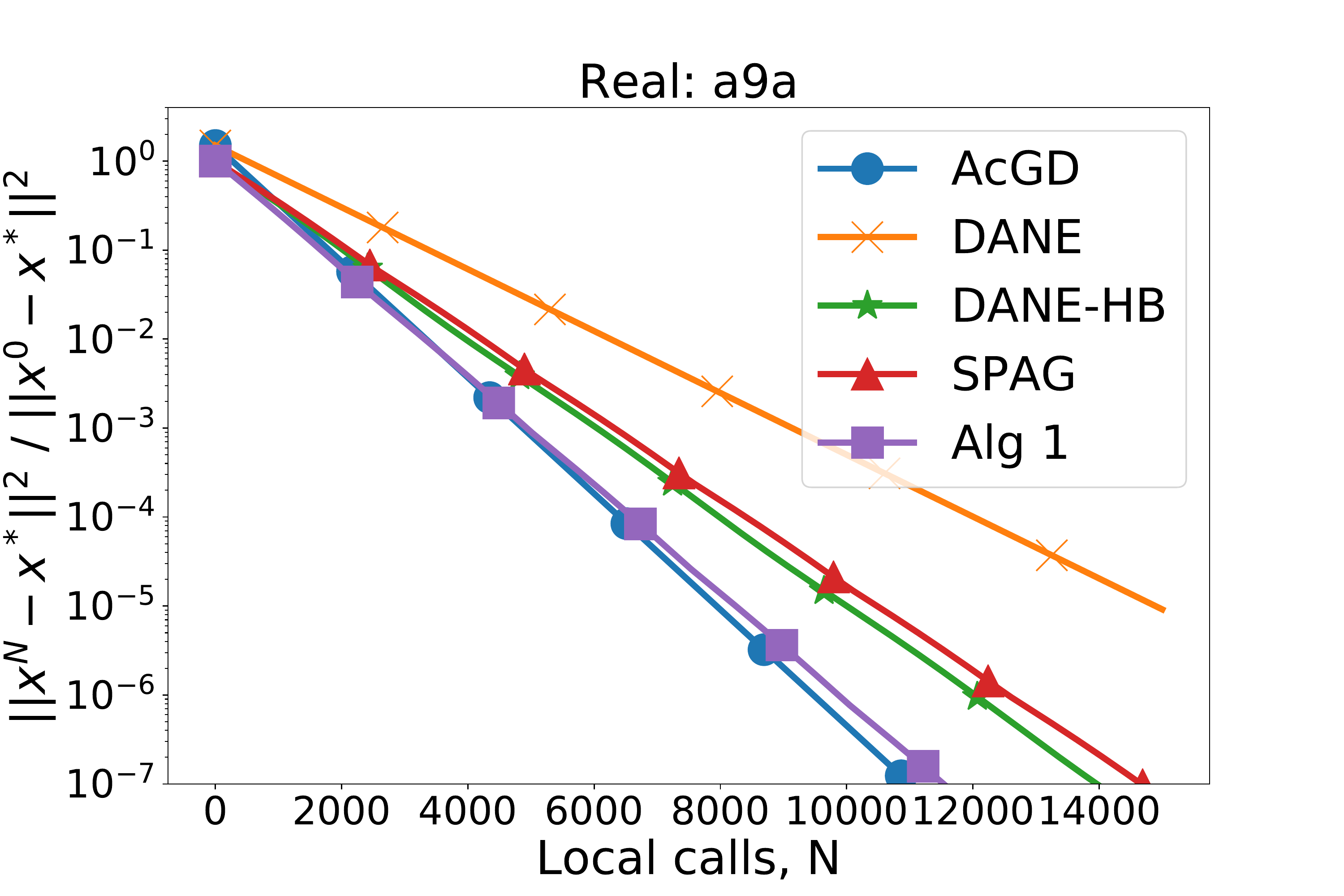}
\end{minipage}%
\caption{\small Ridge regression problem  \eqref{eq:regr}: Comparison of state-of-the-art methods,  under similarity; synthetic data (first two figures on the left) and real data (last to figures from the right).  Distance from optimality vs.    number of communications (first/third panel from the left) and vs.   number of local iterations (second/fourth panel from the left).}
    \label{fig:min}
\end{figure}
The figures show that   our method significantly outperforms   AcGD, and  DANE while compares favorably with   DANE-HB and SPAG  both on communication and local gradient  iterations.

\subsection{Saddle point problems}

Here we consider a modification of \eqref{eq:regr}, the Robust Linear Regression, which leads to the following saddle-point formulation:
\begin{align}\label{eq:rob-regr}
    \textstyle \min_{w} \max_{\norm{r_i}\leq R_r} \frac{1}{2N} \sum\limits_{i=1}^N \left[(w^T (x_i + r_i) - y_i)^2 - \beta \|r_i \|^2\right] + \frac{\lambda}{2} \| w\|^2,
\end{align}
where $r_i$ is the so-called adversarial noise and $\beta>0$ is the regularization associated with it; we set $\lambda = \beta = 0,1$ and $R_r = 0,05$. The network setting and data generation is the same as discussed   in Section \ref{sec:exp_min}.  We compare with the only existing method  for SPPs under similarity, as proposed in   \cite{beznosikov2021distributed}. Results are summarized in Figure \ref{fig:spp}, on synthetic and real data. 

\begin{figure}[h]
\begin{minipage}{0.24\textwidth}
  \centering
\includegraphics[width =  \textwidth ]{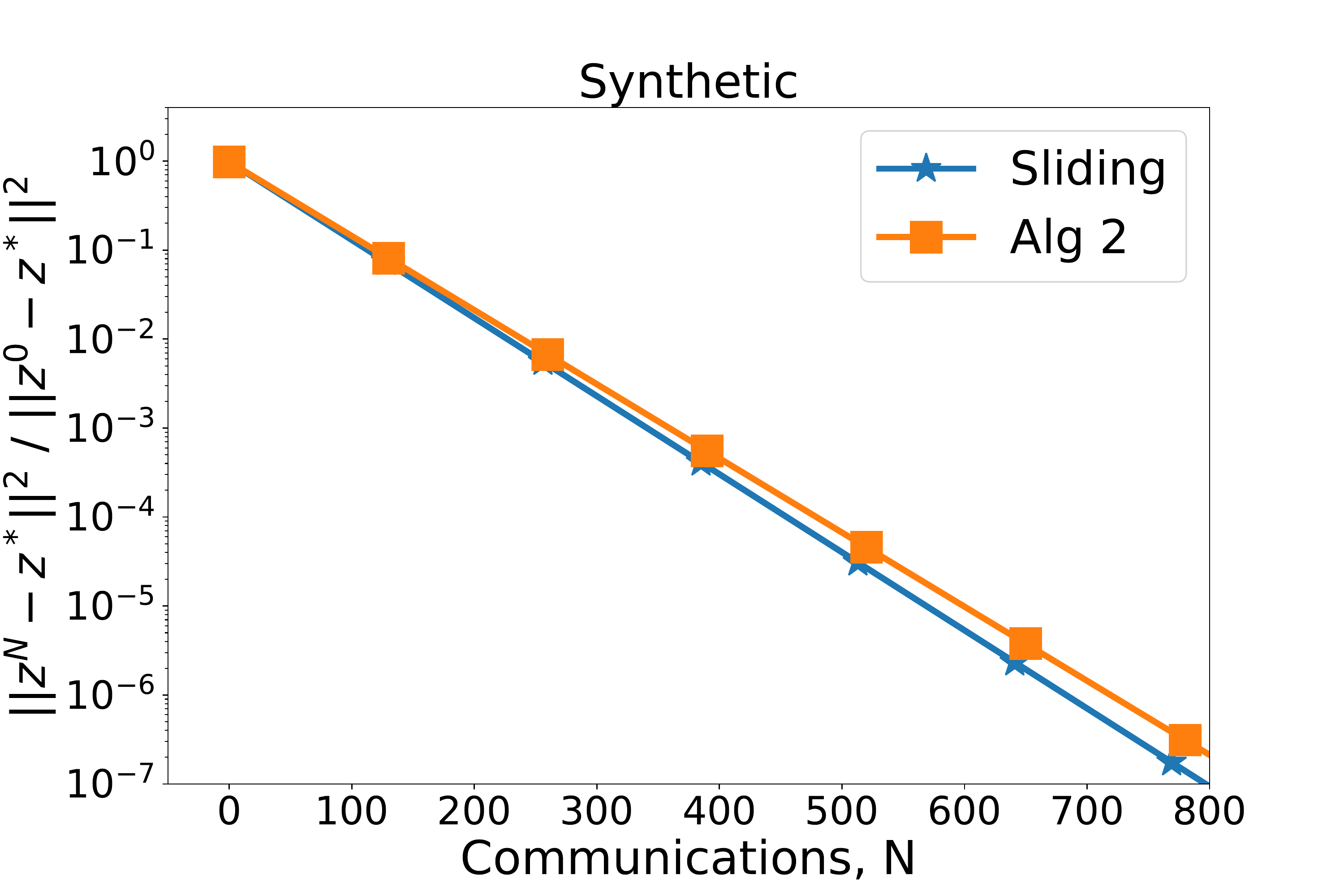}
\end{minipage}%
\begin{minipage}{0.24\textwidth}
  \centering
\includegraphics[width =  \textwidth ]{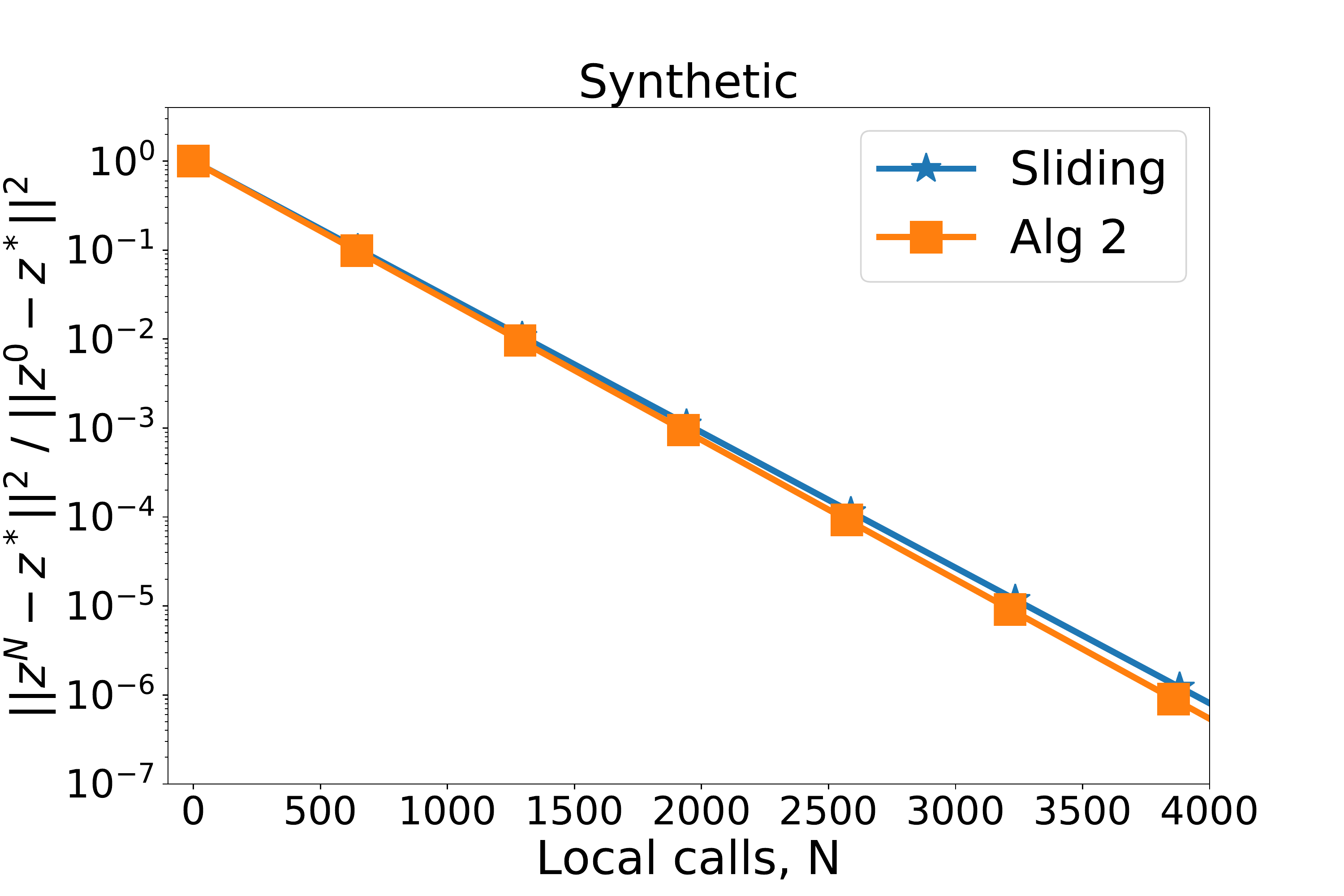}
\end{minipage}%
\begin{minipage}{0.24\textwidth}
  \centering
\includegraphics[width =  \textwidth ]{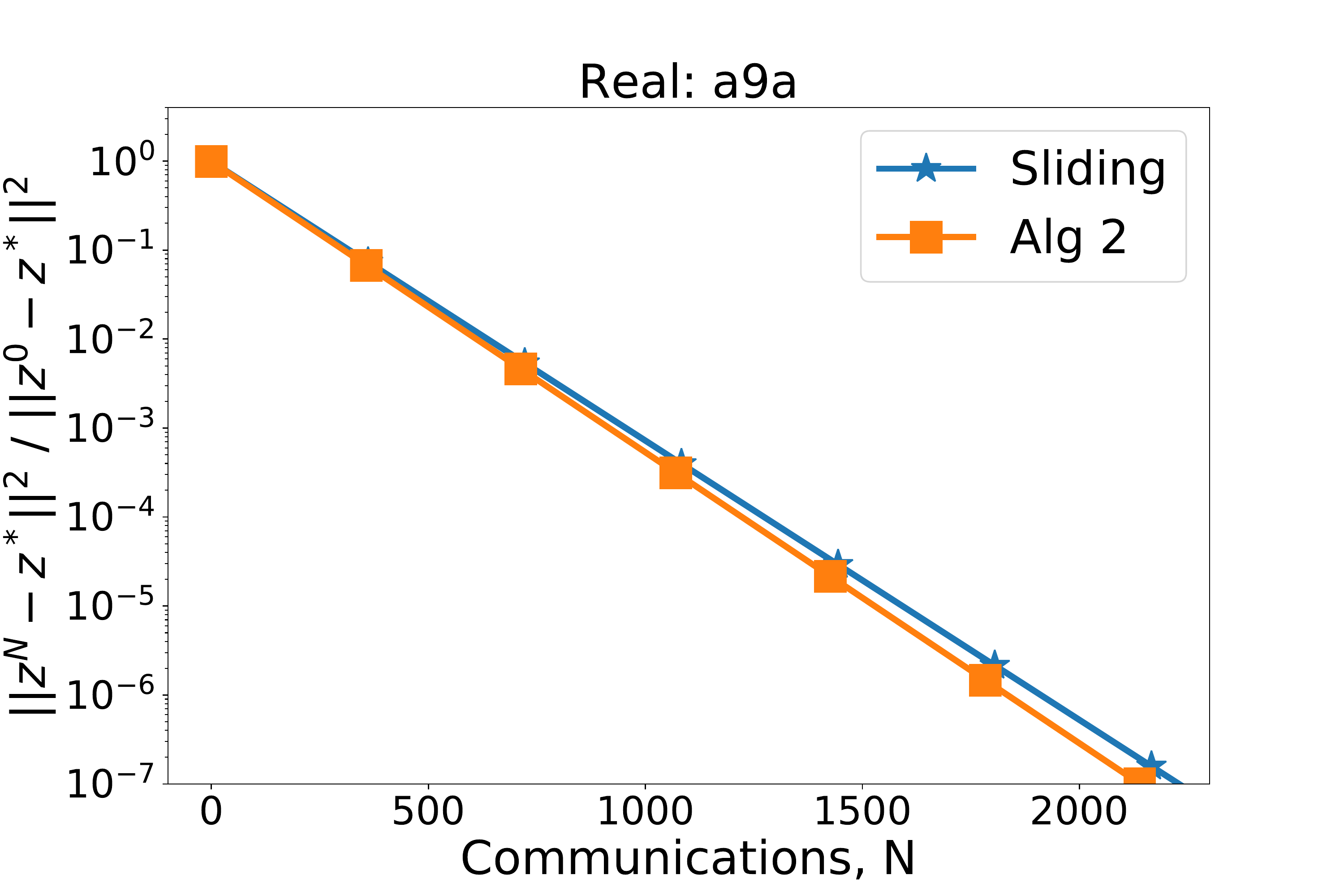}
\end{minipage}%
\begin{minipage}{0.24\textwidth}
  \centering
\includegraphics[width =  \textwidth ]{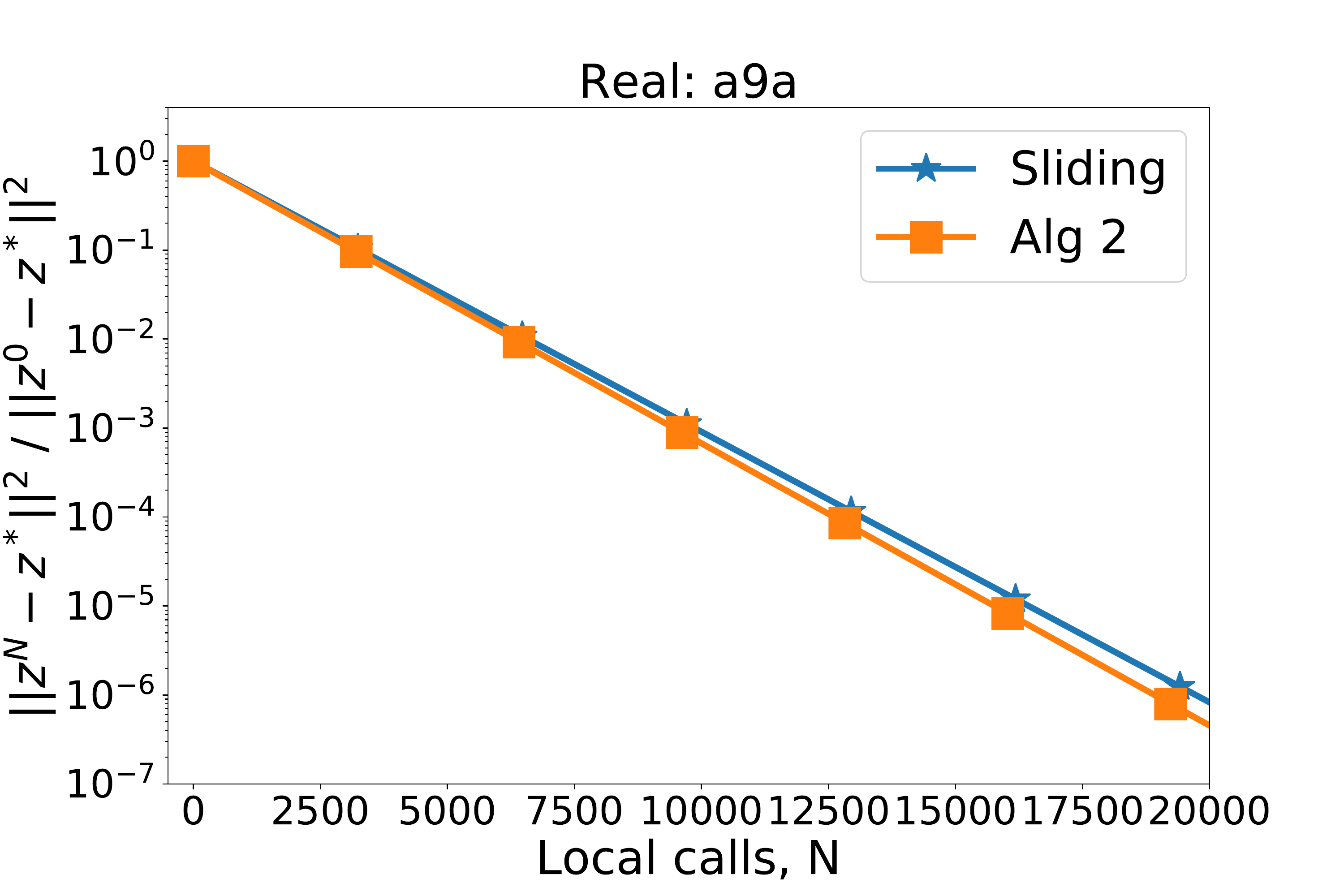}
\end{minipage}%
\caption{\small Robust Linear Regression \eqref{eq:rob-regr}, under similarity assumption: Proposed method vs. Gradient-Sliding;  synthetic data (first two figures on the left) and real data (last to figures from the right). Distance from optimality vs.    number of communications (first/third panel from the left) and vs.   number of local iterations (second/fourth panel from the left).}
    \label{fig:spp}
\end{figure}
It can be seen that our method compares favorably with   \cite{beznosikov2021distributed} both on   communication and gradient  iterations.


\bibliographystyle{plain}
\bibliography{reference.bib}


\newpage

\appendix

\part*{APPENDIX}

\tableofcontents

\newpage

\section{Proofs for Section \ref{sec:slid_min}}

In this section we present a proof of the convergence of Algorithm \ref{ae:alg} in the strongly convex case -- Section \ref{sec:slid_min_strcon}. We also present a modification of Algorithm \ref{ae:alg} for the convex case, as well as a proof of its convergence -- Section \ref{sec:slid_min_conv}.

\subsection{Strongly convex case} \label{sec:slid_min_strcon}
Here we prove Theorem \ref{ae:thm}. First, we need the following lemmas:
\begin{lemma}
	 Consider   \Cref{ae:alg}. Let $\theta$ be defined as in Theorem \ref{ae:thm}: $\theta = \frac{1}{2L_p}$. Then, under Assumptions \ref{ass:mu}-\ref{ass:L_p}, the following inequality holds for all $\bar{x} \in \R^d$
	\begin{equation}
	\label{ae:eq:1}
	\begin{split}
		2\<\bar{x} - x_g^k, \nabla r(x_f^{k+1})>
		\leq&
		2\left[r(\bar{x}) - r(x_f^{k+1})\right] - \mu \sqn{x_f^{k+1} - \bar{x}}
		- \theta\sqn{\nabla r(x_f^{k+1})}
		\\&
		+3\theta\left(\sqn{\nabla A_\theta^k(x_f^{k+1})} - \frac{L_p^2}{3}\sqn{x_g^k- \argmin_{x \in \R^d} A_\theta^k(x)}\right).
	\end{split}
	\end{equation}
\end{lemma}
\begin{proof}
Using $\mu$-strong convexity of $r(x)$, we get
	\begin{align*}
		2\<\bar{x} - x_g^k, \nabla r(x_f^{k+1})>
		=&
		2\<\bar{x} - x_f^{k+1}, \nabla r(x_f^{k+1})>
		+2\<x_f^{k+1} - x_g^k, \nabla r(x_f^{k+1})>
		\\\leq&
		2\left[r(\bar{x}) - r(x_f^{k+1})\right] - \mu \sqn{x_f^{k+1} - \bar{x}}
		+2\<x_f^{k+1} - x_g^k, \nabla r(x_f^{k+1})>
		\\=&
		2\left[r(\bar{x}) - r(x_f^{k+1})\right] - \mu \sqn{x_f^{k+1} - \bar{x}}
		+2\theta\<\theta^{-1}(x_f^{k+1} - x_g^k), \nabla r(x_f^{k+1})>
		\\=&
		2\left[r(\bar{x}) - r(x_f^{k+1})\right] - \mu \sqn{x_f^{k+1} - \bar{x}}
		\\&
		-\frac{1}{\theta}\sqn{x_f^{k+1} - x_g^k} - \theta\sqn{\nabla r(x_f^{k+1})}
		\\&
		+\theta\sqn{\theta^{-1}(x_f^{k+1} - x_g^k) + \nabla r(x_f^{k+1})}.
	\end{align*}
The definition of $A_\theta^k(x)$ and $L_p$-Lipschitzness of $\nabla p$ (Assumption \ref{ass:L_p}) give
	\begin{align*}
		2\<\bar{x} - x_g^k, \nabla r(x_f^{k+1})>
		\leq&
		2\left[r(\bar{x}) - r(x_f^{k+1})\right] - \mu \sqn{x_f^{k+1} - \bar{x}}
		-\frac{1}{\theta}\sqn{x_f^{k+1} - x_g^k} - \theta\sqn{\nabla r(x_f^{k+1})}
		\\&
		+\theta\sqn{\nabla A_\theta^k(x_f^{k+1}) + \nabla p(x_f^{k+1}) - \nabla p (x_g^k)}
		\\\leq&
		2\left[r(\bar{x}) - r(x_f^{k+1})\right] - \mu \sqn{x_f^{k+1} - \bar{x}}
		-\frac{1}{\theta}\sqn{x_f^{k+1} - x_g^k} - \theta\sqn{\nabla r(x_f^{k+1})}
		\\&
		+2\theta\sqn{\nabla A_\theta^k(x_f^{k+1})}
		+2\theta L_p^2\sqn{x_f^{k+1} - x_g^k}
		\\=&
		2\left[r(\bar{x}) - r(x_f^{k+1})\right] - \mu \sqn{x_f^{k+1} - \bar{x}}
		-\frac{1}{\theta}\left(1 - 2\theta^2L_p^2\right) \sqn{x_f^{k+1} - x_g^k}
		\\&
		- \theta\sqn{\nabla r(x_f^{k+1})}
		+2\theta\sqn{\nabla A_\theta^k(x_f^{k+1})}.
	\end{align*}
With $\theta = \frac{1}{2L_p}$, we have
	\begin{align*}
		2\<\bar{x} - x_g^k, \nabla r(x_f^{k+1})>
		\leq&
		2\left[r(\bar{x}) - r(x_f^{k+1})\right] - \mu \sqn{x_f^{k+1} - \bar{x}}
		-\frac{1}{2\theta}\sqn{x_f^{k+1} - x_g^k}
		\\&
		- \theta\sqn{\nabla r(x_f^{k+1})}
		+2\theta\sqn{\nabla A_\theta^k(x_f^{k+1})}
		\\=&
		2\left[r(\bar{x}) - r(x_f^{k+1})\right] - \mu \sqn{x_f^{k+1} - \bar{x}}
		-\frac{1}{4\theta}\sqn{x_g^k- \argmin_{x \in \R^d} A_\theta^k(x)}
		\\&
		+\frac{1}{2\theta}\sqn{x_f^{k+1} - \argmin_{x \in \R^d} A_\theta^k(x)}
		- \theta\sqn{\nabla r(x_f^{k+1})}
		+2\theta\sqn{\nabla A_\theta^k(x_f^{k+1})}.
	\end{align*}
One can observe that $A_\theta^k(x)$ is $\frac{1}{\theta}$-strongly convex. Hence,
	\begin{align*}
		2\<\bar{x} - x_g^k, \nabla r(x_f^{k+1})>
		\leq&
		2\left[r(\bar{x}) - r(x_f^{k+1})\right] - \mu \sqn{x_f^{k+1} - \bar{x}}
		-\frac{1}{4\theta}\sqn{x_g^k- \argmin_{x \in \R^d} A_\theta^k(x)}
		\\&
		+\frac{\theta}{2}\sqn{\nabla A_\theta^k(x_f^{k+1})}
		- \theta\sqn{\nabla r(x_f^{k+1})}
		+2\theta\sqn{\nabla A_\theta^k(x_f^{k+1})}
		\\\leq&
		2\left[r(\bar{x}) - r(x_f^{k+1})\right] - \mu \sqn{x_f^{k+1} - \bar{x}}
		-\frac{1}{4\theta}\sqn{x_g^k- \argmin_{x \in \R^d} A_\theta^k(x)}
		\\&
		+3\theta\sqn{\nabla A_\theta^k(x_f^{k+1})}
		- \theta\sqn{\nabla r(x_f^{k+1})}
		\\=&
		2\left[r(\bar{x}) - r(x_f^{k+1})\right] - \mu \sqn{x_f^{k+1} - \bar{x}}
		- \theta\sqn{\nabla r(x_f^{k+1})}
		\\&
		+3\theta\left(\sqn{\nabla A_\theta^k(x_f^{k+1})} - \frac{1}{12\theta^2}\sqn{x_g^k- \argmin_{x \in \R^d} A_\theta^k(x)}\right)
		\\=&
		2\left[r(\bar{x}) - r(x_f^{k+1})\right] - \mu \sqn{x_f^{k+1} - \bar{x}}
		- \theta\sqn{\nabla r(x_f^{k+1})}
		\\&
		+3\theta\left(\sqn{\nabla A_\theta^k(x_f^{k+1})} - \frac{L_p^2}{3}\sqn{x_g^k- \argmin_{x \in \R^d} A_\theta^k(x)}\right).
	\end{align*}
This completes the proof of Lemma.
\end{proof}

\begin{lemma}
	Consider  \Cref{ae:alg} for Problem \ref{eq:main} under Assumption \ref{ass:mu}-\ref{ass:L_p}, with the following tuning: 
\begin{equation}\label{ae:choice}
    \tau = \min\left\{1,\frac{\sqrt{\mu}}{2\sqrt{L_p}}\right\}, \quad \theta = \frac{1}{2L_p}, \quad  \eta = \min \left\{\frac{1}{2\mu},\frac{1}{2\sqrt{\mu L_p}}\right\}, \quad \alpha = \mu,
	\end{equation}
and let $x_f^{k+1}$   in \cref{ae:line:2} satisfy    
	\begin{equation}\label{aux:grad_app}
	\sqn{\nabla A_\theta^k(x_f^{k+1})} \leq  \frac{L_p^2}{3}\sqn{x_g^k- \argmin_{x \in \R^d} A_\theta^k(x)}.
	\end{equation}
Then, the following inequality holds:
	\begin{equation}\label{ae:rec}
		\frac{1}{\eta}\sqn{x^{k+1} - x^*}
		+
		\frac{2}{\tau}\left[r(x_f^{k+1}) - r(x^*)\right]
		\leq
		\left(1 - \rho\right)
		\left[\frac{1}{\eta}\sqn{x^k - x^*}
		+\frac{2}{\tau}\left[r(x_f^k) - r(x^*)\right]\right],
	\end{equation}
	where  
	\begin{equation}\label{ae:rho}
		\rho := \frac{1}{2}\min \left\{1,\sqrt{\frac{\mu}{L_p}}\right\}.
	\end{equation}
\end{lemma}

\begin{proof}
Using \cref{ae:line:3} of \Cref{ae:alg}, we get
	\begin{align*}
		\frac{1}{\eta}\sqn{x^{k+1} - x^*}
		=&
		\frac{1}{\eta}\sqn{x^k - x^*}
		+\frac{2}{\eta}\<x^{k+1} - x^k, x^k - x^*>
		+\frac{1}{\eta}\sqn{x^{k+1} - x^k}
		\\=&
		\frac{1}{\eta}\sqn{x^k - x^*}
		+2\alpha\<x_f^{k+1} -x^k, x^k - x^*>
		\\&
		-2\<\nabla r(x_f^{k+1}), x^k - x^*>
		+\frac{1}{\eta}\sqn{x^{k+1} - x^k}
		\\=&
		\frac{1}{\eta}\sqn{x^k - x^*}
		+\alpha\sqn{x_f^{k+1} - x^*}
		-\alpha\sqn{x^k - x^*}
		-\alpha\sqn{x_f^{k+1} - x^k}
		\\&
		-2\<\nabla r(x_f^{k+1}), x^k - x^*>
		+\frac{1}{\eta}\sqn{x^{k+1} - x^k}.
	\end{align*}
\Cref{ae:line:1} of \Cref{ae:alg} gives
	\begin{align*}
		\frac{1}{\eta}\sqn{x^{k+1} - x^*}
		=&
		\left(\frac{1}{\eta}-\alpha\right)\sqn{x^k - x^*}
		+\alpha\sqn{x_f^{k+1} - x^*}
		+\frac{1}{\eta}\sqn{x^{k+1} - x^k}
		-\alpha\sqn{x_f^{k+1} - x^k}
		\\&
		+2\<\nabla r(x_f^{k+1}), x^* - x_g^k>
		+\frac{2(1-\tau)}{\tau}\<\nabla r(x_f^{k+1}), x_f^k - x_g^k>.
	\end{align*}
Using \eqref{ae:eq:1} with $\bar{x} = x^*$ and $\bar{x} = x_f^k$, we get
	\begin{align*}
		\frac{1}{\eta}\sqn{x^{k+1} - x^*}
		\leq&
		\left(\frac{1}{\eta}-\alpha\right)\sqn{x^k - x^*}
		+\alpha\sqn{x_f^{k+1} - x^*}
		+\frac{1}{\eta}\sqn{x^{k+1} - x^k}
		-\alpha\sqn{x_f^{k+1} - x^k}
		\\&
		+2\left[r(x^*) - r(x_f^{k+1})\right] - \mu\sqn{x_f^{k+1} - x^*}
		+\frac{2(1-\tau)}{\tau}\left[r(x_f^{k}) - r(x_f^{k+1})\right]
		\\&
		-\frac{1}{2\tau L_p}\sqn{\nabla r(x_f^{k+1})}
		+\frac{3\theta}{\tau}\left(\sqn{\nabla A_\theta^k(x_f^{k+1})} - \frac{L_p^2}{3}\sqn{x_g^k- \argmin_{x \in \R^d} A_\theta^k(x)}\right)
		\\=&
		\left(\frac{1}{\eta}-\alpha\right)\sqn{x^k - x^*}
		+(\alpha - \mu)\sqn{x_f^{k+1} - x^*}
		-\alpha\sqn{x_f^{k+1} - x^k}
		\\&
		+\frac{1}{\eta}\sqn{\eta\alpha(x_f^{k+1} - x^k) - \eta \nabla r(x_f^{k+1})}
		-\frac{1}{2\tau L_p}\sqn{\nabla r(x_f^{k+1})}
		\\&
		+\frac{2(1-\tau)}{\tau}\left[r(x_f^k) - r(x^*)\right]
		-\frac{2}{\tau}\left[r(x_f^{k+1}) - r(x^*)\right]
		\\&
		+\frac{3\theta}{\tau}\left(\sqn{\nabla A_\theta^k(x_f^{k+1})} - \frac{L_p^2}{3}\sqn{x_g^k- \argmin_{x \in \R^d} A_\theta^k(x)}\right)
		\\\leq&
		\left(\frac{1}{\eta}-\alpha\right)\sqn{x^k - x^*}
		+(\alpha - \mu)\sqn{x_f^{k+1} - x^*}
		\\&
		+\alpha(2\eta\alpha - 1)\sqn{x_f^{k+1} - x^k}
		+\left(2\eta - \frac{1}{2\tau L_p}\right)\sqn{\nabla r(x_f^{k+1})}
		\\&
		+\frac{2(1-\tau)}{\tau}\left[r(x_f^k) - r(x^*)\right]
		-\frac{2}{\tau}\left[r(x_f^{k+1}) - r(x^*)\right]
		\\&
		+\frac{3\theta}{\tau}\left(\sqn{\nabla A_\theta^k(x_f^{k+1})} - \frac{L_p^2}{3}\sqn{x_g^k- \argmin_{x \in \R^d} A_\theta^k(x)}\right).
	\end{align*}
The choice of $\alpha,\eta,\tau$ defined by \eqref{ae:choice} gives
	\begin{align*}
		\frac{1}{\eta}\sqn{x^{k+1} - x^*}
		\leq&
		\left(\frac{1}{\eta}-\alpha\right)\sqn{x^k - x^*}
		+\frac{2(1-\tau)}{\tau}\left[r(x_f^k) - r(x^*)\right]
		-\frac{2}{\tau}\left[r(x_f^{k+1}) - r(x^*)\right]
		\\&
		+\frac{3\theta}{\tau}\left(\sqn{\nabla A_\theta^k(x_f^{k+1})} - \frac{L_p^2}{3}\sqn{x_g^k- \argmin_{x \in \R^d} A_\theta^k(x)}\right).
	\end{align*}
With \eqref{aux:grad_app}, we have
	\begin{align*}
		\frac{1}{\eta}\sqn{x^{k+1} - x^*}
		+
		\frac{2}{\tau}\left[r(x_f^{k+1}) - r(x^*)\right]
		&\leq
		\left(\frac{1}{\eta}-\alpha\right)\sqn{x^k - x^*}
		+\frac{2(1-\tau)}{\tau}\left[r(x_f^k) - r(x^*)\right]
		\\&\leq
		\left(1 - \rho\right)
		\left[\frac{1}{\eta}\sqn{x^k - x^*}
		+\frac{2}{\tau}\left[r(x_f^k) - r(x^*)\right]\right],
	\end{align*}
	where $\rho$ is defined by \eqref{ae:rho}.
\end{proof}
To prove Theorem \ref{ae:thm}, it is sufficient to run the recursion \eqref{ae:rec}:
	\begin{equation*}
		\sqn{x^K - x^*} \leq (1-\rho)^K	\left[\sqn{x^0 - x^*}
		+\frac{2\eta}{\tau}\left[r(x^0) - r(x^*)\right]\right] = C(1-\rho)^K,
	\end{equation*}
	where $C$ is defined as
	\begin{equation*}
		C = \sqn{x^0 - x^*}
		+\frac{2\eta}{\tau}\left[r(x^0) - r(x^*)\right].
	\end{equation*}
	Hence, choosing number of iterations $K$ given by \eqref{eq:K} yields
	\begin{equation*}
		\sqn{x^K - x^*} \leq \epsilon.
	\end{equation*}

\subsection{Convex case} \label{sec:slid_min_conv}

The next Algorithm \ref{ae_conv:alg} is an adaptation of Algorithm \ref{ae:alg} for the convex case. In particular, time-varying $\tau_{k+1}$ and $\eta_{k+1}$ are used instead of the momentum $\alpha$. 
\begin{algorithm}[H]
	\caption{Accelerated Extragradient (modification for convex case)}
	\label{ae_conv:alg}
	\begin{algorithmic}[1]
		\State {\bf Input:} $x^0=x_f^0 \in \R^d$
		\State {\bf Parameters:} $K \in \{1,2,\ldots\}, \{\tau_k \}_{k=1}^{K} \subset (0,1]$, $\{\eta_k\}_{k=1}^{K} \subset \R_+,\theta>0$
		\For{$k=0,1,2,\ldots, K-1$}
			\State $x_g^k = \tau_{k+1} x^k + (1-\tau_{k+1}) x_f^k$\label{ae_conv:line:1}
			\State $x_f^{k+1} \approx \argmin_{x \in \R^d}\left[ A_\theta^k(x) \eqdef p(x_g^k) + \<\nabla p(x_g^k),x - x_g^k> + \frac{1}{2\theta}\sqn{x - x_g^k} + q(x)\right]$\label{ae_conv:line:2}
			\State $x^{k+1} = x^k - \eta_{k+1} \nabla r(x_f^{k+1})$\label{ae_conv:line:3}
		\EndFor
		\State {\bf Output:} $x^{K+1}_f$
	\end{algorithmic}
\end{algorithm}

\begin{lemma}
Consider  \Cref{ae_conv:alg} for Problem \ref{eq:main} under Assumption \ref{ass:mu}($\mu=0$)-\ref{ass:L_p}, with the following tuning: 
\begin{equation}\label{ae_conv:choice}
    \tau_k = \frac{2}{k+1}, \quad \theta = \frac{1}{2L_p}, \quad  \eta_k= \frac{1}{2\tau_k L_p},
	\end{equation}
and let $x_f^{k+1}$   in \cref{ae_conv:line:2} satisfy    
	\begin{equation}\label{aux:conv_grad_app}
	\sqn{\nabla A_\theta^k(x_f^{k+1})} \leq  \frac{L_p^2}{3}\sqn{x_g^k- \argmin_{x \in \R^d} A_\theta^k(x)}.
	\end{equation}
Then, the following inequality holds:
	\begin{equation}\label{ae_conv:convergence}
	r(x_f^{k}) - r(x^*) \leq  \frac{4 L_p}{(k+1)^2} \sqn{x^0 - x^*}.
	\end{equation}
\end{lemma}

\begin{proof}
We start from \cref{ae_conv:line:3} of \Cref{ae_conv:alg} and get
	\begin{align*}
		\frac{1}{\eta_{k+1}}\sqn{x^{k+1} - x^*}
		&=
		\frac{1}{\eta_{k+1}}\sqn{x^k - x^*}
		+\frac{2}{\eta_{k+1}}\<x^{k+1} - x^k, x^k - x^*>
		+\frac{1}{\eta_{k+1}}\sqn{x^{k+1} - x^k}
		\\&=
		\frac{1}{\eta_{k+1}}\sqn{x^k - x^*}
		-2\<\nabla r(x_f^{k+1}), x^k - x^*>
		+\frac{1}{\eta_{k+1}}\sqn{x^{k+1} - x^k}.
	\end{align*}
\Cref{ae_conv:line:1} of \Cref{ae_conv:alg} gives
	\begin{align*}
		\frac{1}{\eta_{k+1}}\sqn{x^{k+1} - x^*}
		=&
		\frac{1}{\eta_{k+1}}\sqn{x^k - x^*}
		+\frac{1}{\eta_{k+1}}\sqn{x^{k+1} - x^k}
		\\&
		+2\<\nabla r(x_f^{k+1}), x^* - x_g^k>
		+\frac{2(1-\tau_{k+1})}{\tau_{k+1}}\<\nabla r(x_f^{k+1}), x_f^k - x_g^k>.
	\end{align*}
Using \eqref{ae:eq:1} with $\mu = 0$, $\bar{x} = x^*$ and $\bar{x} = x_f^k$, we get
	\begin{align*}
		\frac{1}{\eta_{k+1}}\sqn{x^{k+1} - x^*}
		\leq&
		\frac{1}{\eta_{k+1}}\sqn{x^k - x^*}
		+\frac{1}{\eta_{k+1}}\sqn{x^{k+1} - x^k}
		+2\left[r(x^*) - r(x_f^{k+1})\right] \\&
		+\frac{2(1-\tau_{k+1})}{\tau_{k+1}}\left[r(x_f^{k}) - r(x_f^{k+1})\right]
		-\frac{1}{2\tau_{k+1} L_p}\sqn{\nabla r(x_f^{k+1})}
		\\&+\frac{3\theta}{\tau_{k+1}}\left(\sqn{\nabla A_\theta^k(x_f^{k+1})} - \frac{L_p^2}{3}\sqn{x_g^k- \argmin_{x \in \R^d} A_\theta^k(x)}\right)
		\\=&\frac{1}{\eta_{k+1}}\sqn{x^k - x^*}
		+\frac{1}{\eta_{k+1}}\sqn{ \eta_k \nabla r(x_f^{k+1})}-\frac{1}{2\tau_{k+1} L_p}\sqn{\nabla r(x_f^{k+1})}
		\\&
		+\frac{2(1-\tau_{k+1})}{\tau_{k+1}}\left[r(x_f^k) - r(x^*)\right]
		-\frac{2}{\tau_{k+1}}\left[r(x_f^{k+1}) - r(x^*)\right]
		\\&
		+\frac{3\theta}{\tau_{k+1}}\left(\sqn{\nabla A_\theta^k(x_f^{k+1})} - \frac{L_p^2}{3}\sqn{x_g^k- \argmin_{x \in \R^d} A_\theta^k(x)}\right)
		\\=&
		\frac{1}{\eta_{k+1}}\sqn{x^k - x^*}
		+\left(\eta_{k+1} - \frac{1}{2\tau_{k+1} L_p}\right)\sqn{\nabla r(x_f^{k+1})}
		\\&
		+\frac{2(1-\tau_{k+1})}{\tau_{k+1}}\left[r(x_f^k) - r(x^*)\right]
		-\frac{2}{\tau_{k+1}}\left[r(x_f^{k+1}) - r(x^*)\right]
		\\&
		+\frac{3\theta}{\tau_{k+1}}\left(\sqn{\nabla A_\theta^k(x_f^{k+1})} - \frac{L_p^2}{3}\sqn{x_g^k- \argmin_{x \in \R^d} A_\theta^k(x)}\right).
	\end{align*}
The choice of $\eta_k$ defined by \eqref{ae_conv:choice} gives
	\begin{align*}
		\sqn{x^{k+1} - x^*}
		\leq&
		\sqn{x^k - x^*}
		+\frac{1-\tau_{k+1}}{\tau_{k+1}^2 L_p}\left[r(x_f^k) - r(x^*)\right]
		-\frac{1}{\tau_{k+1}^2 L_p}\left[r(x_f^{k+1}) - r(x^*)\right]
		\\&
		+\frac{3\theta}{2\tau_{k+1}^2 L_p}\left(\sqn{\nabla A_\theta^k(x_f^{k+1})} - \frac{L_p^2}{3}\sqn{x_g^k- \argmin_{x \in \R^d} A_\theta^k(x)}\right).
	\end{align*}
With \eqref{aux:conv_grad_app}, we have
	\begin{align}\label{ae_conv:lf}
		\sqn{x^{k+1} - x^*}
		+
		\frac{1}{\tau_{k+1}^2 L_p}\left[r(x_f^{k+1}) - r(x^*)\right]
		&\leq
		\sqn{x^k - x^*}
		+\frac{1-\tau_{k+1}}{\tau_{k+1}^2 L_p}\left[r(x_f^k) - r(x^*)\right].
	\end{align}
Let us define $\Psi_k$:
    \begin{align*}
        \Psi_k := \sqn{x^{k} - x^*}
		+
		\frac{1}{\tau_{k}^2 L_p}\left[r(x_f^{k}) - r(x^*)\right].
    \end{align*}
Using \eqref{ae_conv:lf}, $\Psi_k$ defined above and $\tau_k$ defined by \eqref{ae_conv:choice} we get:
    \begin{align*}
        \frac{1}{\tau_{k+1}L_p}\left[r(x_f^{k+1}) - r(x^*)\right] &\leq \Psi_{k+1} \\&\leq \sqn{x^k - x^*}
		+\frac{1-\tau_{k+1}}{\tau_{k+1}^2 L_p}\left[r(x_f^k) - r(x^*)\right]
		\\&= \sqn{x^k - x^*}
		+\frac{(k+2)^2-2(k+2)}{4 L_p}\left[r(x_f^k) - r(x^*)\right]
		\\&= \sqn{x^k - x^*}
		+\frac{(k+2)k}{4 L_p}\left[r(x_f^k) - r(x^*)\right]
		\\&\leq \sqn{x^k - x^*}
		+\frac{(k+1)^2}{4 L_p}\left[r(x_f^k) - r(x^*)\right]
		\\&= \sqn{x^k - x^*}
		+\frac{1}{\tau_k^2 L_p}\left[r(x_f^k) - r(x^*)\right] = \Psi_k.
    \end{align*}
Next, we apply the previous inequality
    \begin{align*}
       \frac{1}{\tau_{k+1}L_p}\left[r(x_f^{k+1}) - r(x^*)\right]\leq \Psi_{k+1} \leq  \Psi_{k}	\leq \ldots \leq \Psi_{1} \leq \sqn{x^0 - x^*}
		+\frac{1-\tau_{1}}{\tau_{1}^2 L_p}\left[r(x_f^1) - r(x^*)\right].  
    \end{align*}
With $\tau_1 = 1$, we have 
    \begin{align*}
       \frac{1}{\tau_{k+1}L_p}\left[r(x_f^{k+1}) - r(x^*)\right]\leq \sqn{x^0 - x^*}.  
    \end{align*}
Finally, again with the choice of $\tau_k$ defined by \eqref{ae_conv:choice}, we get \eqref{ae_conv:convergence}.
\end{proof}

Using \eqref{ae_conv:convergence}, we get 
\begin{align*}
    r(x_f^{T}) - r(x^*) \leq \varepsilon
\end{align*}
after 
\begin{align*}
    T = \sqrt{\frac{4L_p }{\varepsilon}}\|x^0 - x^*\|
\end{align*}
iterations of \Cref{ae_conv:alg}. This is what Theorem \ref{th4} is about.

\section{Proofs for Section \ref{sec:slid_VI}} \label{sec:slid_VI_proof}
In this section we present a proof of the convergence of Algorithm \ref{ae:algVI} in the strongly monotone case -- Section \ref{sec:slid_VI_strmon}. We also present a modification of Algorithm \ref{ae:alg} for the monotone case, as well as a proof of its convergence -- Section \ref{sec:slid_VI_mon}.

\subsection{Strongly monotone case} \label{sec:slid_VI_strmon}
Here we prove Theorem \ref{ae:thmVI}. First, we need the following lemmas:
\begin{lemma}
	Consider   \Cref{ae:algVI}. Let $\theta$ be defined as in Theorem \ref{ae:thmVI}: $\theta = \frac{1}{2L_p}$. Then, under Assumptions \ref{ass:muVI}-\ref{ass:L_pVI}, the following inequality holds for all $\bar{x} \in \R^d$
	\begin{equation}
	\begin{split}\label{ae:eq:1VI}
		2\<x^* - x^k, R(u^k)>
		\leq&
		- 2\mu \sqn{u^k - x^*}
		- \theta\sqn{R(u^k)}
		\\&
		+3\theta\left(\sqn{B_\theta^k(u^k)} - \frac{L_p^2}{3}\sqn{x^k- \tilde u^k}\right).
	\end{split}
	\end{equation}
\end{lemma}
\begin{proof}
Using property of the solution: $R(x^*) = 0$ and $\mu$-strong monotonicity of $R(x)$, we get
	\begin{align*}
		2\<x^* - x^k, R(u^k)>
		=&
		2\<x^* - u^k, R(u^k)>
		+2\<u^k - x^k, R(u^k)>
		\\\leq&
		2\<x^* - u^k, R(u^k) - R(x^*)>
		+2\<u^k - x^k, R(u^k)>
		\\\leq&
		- 2\mu \sqn{u^k - x^*}
		+2\<u^k - x^k, R(u^k)>
		\\=&
		- 2\mu \sqn{u^k - x^*}
		+2\theta\<\theta^{-1}(u^k - x^k), R(u^k)>
		\\=&
	    - 2\mu \sqn{u^k - x^*}
		\\&
		-\frac{1}{\theta}\sqn{u^k - x^k} - \theta\sqn{R(u^k)}
		+\theta\sqn{\theta^{-1}(u^k - x^k) + R(u^k)}.
	\end{align*}
The definition of $B_\theta^k(x)$ and $L_p$-Lipschitzness of $P$ (Assumption \ref{ass:L_pVI}) give
	\begin{align*}
		2\<x^* - x^k, R(u^k)>
		\leq&
	    - 2\mu \sqn{u^k - x^*}
		-\frac{1}{\theta}\sqn{u^k - x^k} - \theta\sqn{R(u^k)}
		\\&
		+\theta\sqn{B_\theta^k(u^k) + P(u^k) - P (x^k)}
		\\\leq&
		- 2\mu \sqn{u^k - x^*}
		-\frac{1}{\theta}\sqn{u^k - x^k} - \theta\sqn{R(u^k)}
		\\&
		+2\theta\sqn{B_\theta^k(u^k)}
		+2\theta L_p^2\sqn{u^k - x^k}
		\\=&
		- 2\mu \sqn{u^k - x^*}
		-\frac{1}{\theta}\left(1 - 2\theta^2L_p^2\right) \sqn{u^k - x^k}
		\\&
		- \theta\sqn{R(u^k)}
		+2\theta\sqn{B_\theta^k(u^k)}.
	\end{align*}
With $\theta = \frac{1}{2L_p}$, we have
	\begin{align*}
		2\<x^* - x^k, R(u^k)>
		\leq&
		- 2\mu \sqn{u^k - x^*}
		-\frac{1}{2\theta}\sqn{u^k - x^k}
		- \theta\sqn{R(u^k)}
		+2\theta\sqn{B_\theta^k(u^k)}
		\\=&
		- 2\mu \sqn{u^k - \bar{x}}
		-\frac{1}{4\theta}\sqn{x^k- \tilde u^k}
		\\&
		+\frac{1}{2\theta}\sqn{u^k - \tilde u^k}
		- \theta\sqn{R(u^k)}
		+2\theta\sqn{B_\theta^k(u^k)}.
	\end{align*}
	One can observe that $B_\theta^k(x)$ is $\frac{1}{\theta}$-strongly monotone. It gives that
	$$
	\frac{1}{\theta} \|x - y \|^2 \leq \langle B_\theta^k(x) - B_\theta^k(y); x -y \rangle \leq \| B_\theta^k(x) - B_\theta^k(y) \| \cdot \| x -y  \|,
	$$
	and with $B_\theta^k(\tilde u^k) = 0$ (since $\tilde u^k$ is the solution of line \ref{ae:line:2VI}), we get
	$$
	\frac{1}{\theta^2} \|u^k - \tilde u^k \|^2 \leq \| B_\theta^k(u^k) - B_\theta^k(\tilde u^k) \|^2 = \| B_\theta^k(u^k) \|^2.
	$$
	Hence,
	\begin{align*}
		2\<x^* - x^k, R(u^k)>
		\leq&
		- 2\mu \sqn{u^k - x^*}
		-\frac{1}{4\theta}\sqn{x^k- \tilde u^k}
		\\&
		+\frac{\theta}{2}\sqn{B_\theta^k(u^k)}
		- \theta\sqn{R(u^k)}
		+2\theta\sqn{B_\theta^k(u^k)}
		\\\leq&
		- 2\mu \sqn{u^k - \bar{x}}
		-\frac{1}{4\theta}\sqn{x^k- \tilde u^k}
		\\&
		+3\theta\sqn{B_\theta^k(u^k)}
		- \theta\sqn{R(u^k)}
		\\=&
		- 2\mu \sqn{u^k - \bar{x}}
		- \theta\sqn{R(u^k)}
		\\&
		+3\theta\left(\sqn{B_\theta^k(u^k)} - \frac{1}{12\theta^2}\sqn{x^k-\tilde u^k}\right)
		\\=&
	    - 2\mu \sqn{u^k - \bar{x}}
		- \theta\sqn{R(u^k)}
		\\&
		+3\theta\left(\sqn{B_\theta^k(u^k)} - \frac{L_p^2}{3}\sqn{x^k- \tilde u^k}\right).
	\end{align*}
This completes the proof of Lemma.
\end{proof}

\begin{lemma}
Consider  \Cref{ae:algVI} for Problem \ref{eq:mainVI} under Assumption \ref{ass:muVI}-\ref{ass:L_pVI}, with the following tuning:
	\begin{equation}\label{ae:choiceVI}
    \theta = \frac{1}{2L_p}, \quad  \eta = \min \left\{\frac{1}{4\mu},\frac{1}{4 L_p}\right\}, \quad \alpha = 2\mu,
	\end{equation}
and let $u^k$ in \cref{ae:line:2VI} satisfies 
	\begin{equation}\label{aux:gradVI_app}
	\sqn{B_\theta^k(u^k)} \leq  \frac{L_p^2}{3}\sqn{x^k- \tilde u^k}.
	\end{equation}
Then, the following inequality holds:
	\begin{equation}\label{ae:recVI}
		\sqn{x^{k+1} - x^*}
		\leq
		\left(1 - 2 \mu \eta \right)^K \sqn{x^0 - x^*}.
	\end{equation}
\end{lemma}

\begin{proof}
Using \cref{ae:line:3VI} of \Cref{ae:algVI}, we get
	\begin{align*}
		\sqn{x^{k+1} - x^*}
		=&
		\sqn{x^k - x^*}
		+2\<x^{k+1} - x^k, x^k - x^*>
		+\sqn{x^{k+1} - x^k}
		\\=&
	    \sqn{x^k - x^*}
		+2\eta \alpha\<u^k -x^k, x^k - x^*>
		\\&
		-2\eta\<R(u^k), x^k - x^*>
		+\sqn{x^{k+1} - x^k}
		\\=&
		\sqn{x^k - x^*}
		+\eta\alpha\sqn{u^k - x^*}
		-\eta\alpha\sqn{x^k - x^*}
		-\eta\alpha\sqn{u^k - x^k}
		\\&
		-2\eta \<R(u^k), x^k - x^*>
		+\sqn{x^{k+1} - x^k}.
	\end{align*}
With \eqref{ae:eq:1VI}, we get
	\begin{align*}
		\sqn{x^{k+1} - x^*}
		\leq&
		\left(1-\eta\alpha\right)\sqn{x^k - x^*}
		+\eta\alpha\sqn{u^k - x^*}
		+\sqn{x^{k+1} - x^k}
		-\eta\alpha\sqn{u^k - x^k}
		\\&
		- 2\eta\mu \sqn{u^k - x^*}
		- \eta\theta\sqn{R(u^k)}
		\\&
		+3\eta\theta\left(\sqn{B_\theta^k(u^k)} - \frac{L_p^2}{3}\sqn{x^k- \tilde u^k}\right)
		\\=&
		\left(1-\eta\alpha\right)\sqn{x^k - x^*}
		+\sqn{\eta\alpha(u^k - x^k) - \eta R(u^k)}
		-\eta\alpha\sqn{u^k - x^k}
		\\&
		- \eta (2\mu - \alpha ) \sqn{u^k - x^*}
		- \eta\theta\sqn{R(u^k)}
		\\&
		+3\eta\theta\left(\sqn{B_\theta^k(u^k)} - \frac{L_p^2}{3}\sqn{x^k- \tilde u^k}\right)
		\\\leq&
		\left(1-\eta\alpha\right)\sqn{x^k - x^*}
		-\eta\alpha (1 - 2\eta\alpha)\sqn{u^k - x^k}
		\\&
		- \eta (2\mu - \alpha ) \sqn{u^k - x^*}
		- \eta (\theta - 2 \eta)\sqn{R(u^k)}
		\\&
		+3\eta\theta\left(\sqn{B_\theta^k(u^k)} - \frac{L_p^2}{3}\sqn{x^k- \tilde u^k}\right).
	\end{align*}
The choice of $\alpha,\eta, \theta$ defined by \eqref{ae:choiceVI} gives
	\begin{align*}
		\sqn{x^{k+1} - x^*}
		&\leq
		\left(1-2\eta\mu\right)\sqn{x^k - x^*}
		+3\eta\theta\left(\sqn{B_\theta^k(u^k)} - \frac{L_p^2}{3}\sqn{x^k- \tilde u^k}\right).
	\end{align*}
	Using \eqref{aux:gradVI} we get
	\begin{align*}
		\sqn{x^{k+1} - x^*}
		&\leq
		\left(1-2\eta\mu\right)\sqn{x^k - x^*}.
	\end{align*}
\end{proof}
To prove Theorem \ref{ae:thmVI}, it is sufficient to run the recursion \eqref{ae:recVI}:
\begin{equation*}
		\sqn{x^K - x^*} \leq (1-2\eta \mu)^K	\left[\sqn{x^0 - x^*}
		+\frac{2\eta}{\tau}\left[r(x^0) - r(x^*)\right]\right] = C(1-\rho)^K,
	\end{equation*}
Hence, choosing number of iterations $K$ given by \eqref{eq:KVI} yields
	\begin{equation*}
		\sqn{x^K - x^*} \leq \epsilon.
	\end{equation*}
	
\subsection{Monotone case}\label{sec:slid_VI_mon}

The next Algorithm \ref{ae:algVI} is an adaptation of Algorithm \ref{ae:algVIm} for the monotone case. In particular, we remove the momentum $\alpha$. 

\begin{algorithm}[H]
	\caption{Extragradient (modification for monotone case)}
	\label{ae:algVIm}
	\begin{algorithmic}[1]
		\State {\bf Input:} $x^0 \in \R^d$
		\State {\bf Parameters:} $\eta,\theta > 0, K \in \{1,2,\ldots\}$
		\For{$k=0,1,2,\ldots, K-1$}
			\State Find $u^k \approx \tilde u^k$ where $\tilde u^k$ is a solution for
			\vspace{-0.2cm}
			$$
			\text{Find } \tilde u^k \in \R^d ~:~ B_\theta^k(\tilde u^k)= 0 \text{ with } B_\theta^k(x) \eqdef P(x^k) + Q(x) + \frac{1}{\theta} (x - x^k)
			$$
			\label{ae:line:2VIm}
			\vspace{-0.5cm}
			\State $x^{k+1} = x^k - \eta R(u^k)$\label{ae:line:3VIm}
		\EndFor
		\State {\bf Output:} $x^K$
	\end{algorithmic}
\end{algorithm}

\begin{lemma}
Consider  \Cref{ae:algVIm} for Problem \ref{eq:mainVI} under Assumption \ref{ass:muVI}($\mu=0$)-\ref{ass:L_pVI}, with the following tuning:
\begin{equation}\label{ae_conv:choiceVIm}
    \theta = \frac{1}{2L_p}, \quad  \eta= \frac{1}{4 L_p},
	\end{equation}
and let $u^k$ in \cref{ae:line:2VIm} satisfies 
	\begin{equation}\label{aux:gradVIm_app}
	\sqn{B_\theta^k(u^k)} \leq  \frac{L_p^2}{3}\sqn{x^k- \tilde u^k}.
	\end{equation}
Then, the following inequality holds:
	\begin{equation}\label{ae:recVIm}
		\sup_{x \in \mathcal{C}}\<R(x), \left(\frac{1}{K}\sum\limits_{k=0}^{K-1} u^k \right) - x> 
		\leq \frac{2L_p \sqn{x^0 - x}}{K}.
	\end{equation}
\end{lemma}
\begin{remark}
Here we do not take the maximum over the entire set $\R^d$ (as in the classical version for VIs \cite{juditsky2011solving}), but over $\mathcal{C}$ -- a compact subset of $\R^d$. Thus, we can also consider unbounded sets $\R^d$. This is permissible, since such a version of the criterion is valid if the solution $z^{*}$ lies in $\mathcal{C}$; for details see the work of \cite{nesterov2007dual}.
\end{remark}

\begin{proof}
We start from \cref{ae:line:3VIm} of \Cref{ae:algVIm} and get
	\begin{align*}
		\sqn{x^{k+1} - x}
		&=
		\sqn{x^k - x}
		+2\<x^{k+1} - x^k, x^k - x>
		+\sqn{x^{k+1} - x^k}
		\\&=
	    \sqn{x^k - x}
		-2\eta\<R(u^k), x^k - x>
		+\sqn{x^{k+1} - x^k}.
	\end{align*}
Using \eqref{ae:eq:1VI} with $\mu = 0$, we get
	\begin{align*}
		\sqn{x^{k+1} - x}
		\leq&
		\sqn{x^k - x}
		+\sqn{x^{k+1} - x^k}
		\\&
		-2\eta\<R(u^k), u^k - x>
		- \eta\theta\sqn{R(u^k)}
		\\&
		+3\eta\theta\left(\sqn{B_\theta^k(u^k)} - \frac{L_p^2}{3}\sqn{x^k- \tilde u^k}\right)
		\\=&
		\sqn{x^k - x} -2\eta\<R(u^k), u^k - x>
		+\eta^2\sqn{ R(u^k)} - \eta\theta\sqn{R(u^k)}
		\\&
		+3\eta\theta\left(\sqn{B_\theta^k(u^k)} - \frac{L_p^2}{3}\sqn{x^k- \tilde u^k}\right)
		\\\leq&
		\sqn{x^k - x} -2\eta\<R(u^k), u^k - x>
		- \eta (\theta - \eta)\sqn{R(u^k)}
		\\&
		+3\eta\theta\left(\sqn{B_\theta^k(u^k)} - \frac{L_p^2}{3}\sqn{x^k- \tilde u^k}\right).
	\end{align*}
The choice of $\theta, \eta$ defined by \eqref{ae_conv:choice} give
	\begin{align*}
		\sqn{x^{k+1} - x}
		\leq&
		\sqn{x^k - x} -2\eta\<R(u^k), u^k - x> 
		\\&
		+3\eta\theta\left(\sqn{B_\theta^k(u^k)} - \frac{L_p^2}{3}\sqn{x^k- \tilde u^k}\right).
	\end{align*}
With \eqref{aux:gradVIm_app}, we have
	\begin{align*}
			\sqn{x^{k+1} - x}
		&\leq
		\sqn{x^k - x} -2\eta\<R(u^k), u^k - x> .
	\end{align*}
Summing from $k = 0$ to $K-1$, we obtain
    \begin{align*}
		\frac{1}{K}\sum\limits_{k=0}^{K-1}\<R(u^k), u^k - x> 
		&\leq
		\frac{\sqn{x^0 - x} - \sqn{x^{K} - x}}{2\eta K}
		\\&\leq
		\frac{\sqn{x^0 - x}}{2\eta K}.
	\end{align*}
Monotonicity of $R$ gives
    \begin{align*}
		\<R(x), \left(\frac{1}{K}\sum\limits_{k=0}^{K-1} u^k \right) - x> &= \frac{1}{K}\sum\limits_{k=0}^{K-1}\<R(x), u^k - x>  \\
		&\leq \frac{1}{K}\sum\limits_{k=0}^{K-1}\<R(u^k), u^k - x> \\ 
		&\leq 
		\frac{\sqn{x^0 - x}}{2\eta K} \leq \frac{2L_p \sqn{x^0 - x}}{K}.
	\end{align*}
By taking the supremum over the set $\mathcal{C}$, we get
    \begin{align*}
		\sup_{x \in \mathcal{C}}\<R(x), \left(\frac{1}{K}\sum\limits_{k=0}^{K-1} u^k \right) - x> 
		&\leq \frac{2L_p \sqn{x^0 - x}}{K}.
	\end{align*}
\end{proof}
Using \eqref{ae:recVIm}, we get 
\begin{align*}
    \sup_{x \in \mathcal{C}}\<R(x), \left(\frac{1}{K}\sum\limits_{k=0}^{K-1} u^k \right) - x>  \leq \varepsilon
\end{align*}
after 
\begin{align*}
    T = \frac{2L_p }{\varepsilon}\|x^0 - x^*\|^2
\end{align*}
iterations of \Cref{ae:algVIm}. This is what Theorem \ref{th10} is about.

\newpage 

\section{Experiment details}

The numerical experiments are run on a machine with 8 Intel Core(TM) i7-9700KF 3.60GHz CPU cores with 64GB RAM. The methods are implemented in Python 3.7 using NumPy and SciPy.


In this section, we estimate the smoothness, strong convexity as well as the similarity parameters for objective \eqref{eq:regr}. We denote the identity matrix as $I$ (with the sizes determined by the context). Given a set of data points $X = (x_1\ldots x_N)^\top\in\R^{N\times d}$ and an associated set of labels $y = (y_1\ldots y_N)^\top\in\R^N$, the Linear Regression problem \eqref{eq:regr} is
\begin{align*}
    \min_{\norm{w}}  g(w) := \frac{1}{2N} \sum\limits_{i=1}^N (w^T x_i - y_i)^2 + \frac{\lambda}{2} \| w\|^2.
\end{align*}
Equivalently, $g(w)$ can be expressed as
\begin{align*}
    g(w) = \frac{1}{2N} \norm{Xw - y}^2 + \frac{\lambda}{2}\sqn{w},
\end{align*}
and its gradient writes as
\begin{align*}
    \nabla g(w)
    &= \frac{1}{N} \cbraces{X^\top Xw - X^\top y } \lambda w.
\end{align*}
The Hessian of $g(w)$ is
\begin{align*}
    \nabla^2 g(w) &= \frac{1}{N} X^\top X + \lambda I.
\end{align*}
We are now ready to estimate the spectrum of the Hessian
\begin{align*}
    \norm{\nabla^2 g(w) v} 
    &\leq \frac{1}{N}\lambda_{\max}(X^\top X) \norm{v}+ \lambda\norm{v} \\
    &\leq \cbraces{\frac{1}{N} \lambda_{\max}(X^\top X) + \lambda} \cdot \norm{v} =: L_g \norm{v}.
\end{align*}
Therefore, we can estimate the Lipschitz constant of $\nabla g(w)$ as $L_g$. The same way we can estimate all $L_i$ and take final $L= \max(L_g, L_1, \ldots, L_n)$.

Let us discuss the bound on the similarity parameter. Given two datasets $\braces{X\in\R^{N\times d},~ y\in\R^N}$ and $\braces{\widetilde X\in\R^{\widetilde N\times d},~ \widetilde y\in\R^{\widetilde N}}$, we define
\begin{align*}
    \widetilde g(w) = \frac{1}{2\widetilde N} \norm{\widetilde Xw - \tilde y}^2 + \frac{\lambda}{2}\sqn{w}.
\end{align*}
And then the similarity coefficient $\delta^{g, \widetilde g}$ between functions $g$ and $\widetilde g$ is
\begin{align*}
    \delta^{g, \widetilde g}&= \lambda_{\max}\cbraces{\frac{1}{N} X^\top X - \frac{1}{\widetilde N} \widetilde X^\top \widetilde X}.
\end{align*}
Hence, we can take $\delta = \max(\delta^{g, g_1}, \ldots, \delta^{g, g_n})$.

Finally, we estimate the strong convexity parameter as $\mu = \lambda$.

A way of estimating the constants of strongly monotonicity, smoothness and similarity for \eqref{eq:rob-regr} is given in Section C of \cite{beznosikov2021distributed}.

\end{document}